\pdfoutput=1
\documentclass[a4paper,USenglish,cleveref, autoref, thm-restate]{lipics-v2021}

\hideLIPIcs
\nolinenumbers

\usepackage[noadjust]{cite}
\usepackage{booktabs}

\usepackage{tikz}
\usetikzlibrary{positioning,automata,fit,shapes,calc}
\usepackage{dsfont}
\usepackage{thmtools,thm-restate}
\usepackage{pgfplots}
\usepgfplotslibrary{fillbetween}
\pgfplotsset{width=10cm,compat=1.9}
\usepackage{tabularx}
    \newcolumntype{L}{>{\raggedright\arraybackslash}X}
    \newcolumntype{R}{>{\raggedleft\arraybackslash}X}
\usepackage{wrapfig}

\usepackage{thm-restate}

\theoremstyle{definition}

\theoremstyle{plain}

\theoremstyle{definition}

\theoremstyle{plain}

\theoremstyle{plain}

\theoremstyle{plain}

\theoremstyle{definition}

\newtheorem*{theorem*}{Main result}

\usepackage{amssymb}
\usepackage{tikz}

\newcommand{\rawdiaplus}{%
  \begin{tikzpicture}
    \useasboundingbox (-0.7ex, -0.9ex) rectangle (0.7ex, 0.9ex);
    \node (w) at (-0.7ex,0) {};
    \node (e) at (+0.7ex,0) {};
    \node (s) at (0,-0.9ex) {};
    \node (n) at (0,+0.9ex) {};
    \draw (n.center) -- (e.center) -- (s.center) -- (w.center) -- (n.center);
    \draw (n.center) -- (s.center);
    \draw (e.center) -- (w.center);
  \end{tikzpicture}}

\newsavebox{\diamondplusbox}
\savebox{\diamondplusbox}{\rawdiaplus}

\newcommand{\rawdiaminus}{%
  \begin{tikzpicture}
    \useasboundingbox (-0.7ex, -0.9ex) rectangle (0.7ex, 0.9ex);
    \node (w) at (-0.7ex,0) {};
    \node (e) at (+0.7ex,0) {};
    \node (s) at (0,-0.9ex) {};
    \node (n) at (0,+0.9ex) {};
    \draw (n.center) -- (e.center) -- (s.center) -- (w.center) -- (n.center);
    \draw (e.center) -- (w.center);
  \end{tikzpicture}}

\newsavebox{\diamondminusbox}
\savebox{\diamondminusbox}{\rawdiaminus}

\newcommand{\Var}{\mathbb{V}}
\newcommand{\VP}{\mathbb{VPE}}
\newcommand{\VPE}{\mathbb{VPE}}

\newcommand{\tudparagraph}[2]{%
\vspace*{#1}
\noindent
{\bf #2}
}

\newcommand{\cE}{\mathcal{E}}

\newcommand{\cM}{\mathcal{M}}

\newcommand{\eqdef}{\ensuremath{\stackrel{\text{\tiny def}}{=}}}

\newcommand{\Ende}{\hfill ${\scriptscriptstyle \blacksquare}$}

\renewcommand{\Pr}{\mathrm{Pr}}

\newcommand{\after}[2]{\residual{#1}{#2}}

\newcommand{\sinit}{s_{\mathit{\scriptscriptstyle init}}}

\newcommand{\Act}{\mathit{Act}}

\newcommand{\act}{\alpha}

\newcommand{\fpath}{\pi}

\newcommand{\last}{\mathit{last}}

\newcommand{\prefix}[2]{\mathit{pref}(#1,#2)}

\newcommand{\Cyl}{\mathit{Cyl}}

\newcommand{\sched}{\mathfrak{S}}
\newcommand{\tsched}{\mathfrak{T}}
\newcommand{\msched}{\mathfrak{M}}

\newcommand{\vsched}{\mathfrak{V}}

\newcommand{\rsched}{\mathfrak{R}}

\newcommand{\Min}{\mathfrak{Min}}

\newcommand{\residual}[2]{#1 {\uparrow} {#2}}

\newcommand{\wgt}{\mathit{wgt}}

\newcommand{\goal}{\mathit{goal}}

\newcommand{\Rational}{\mathbb{Q}}

\newcommand{\CiteAppendix}[1]{}

\newcommand{\MeanPayoff}{\mathbb{MP}}

\author{Jakob Piribauer}{Technische Universit\"at Dresden, Germany}{jakob.piribauer@tu-dresden.de}{0000-0003-4829-0476}{}%{email}{orcid}{funding}

\author{Ocan Sankur}{Univ Rennes, Inria, CNRS, IRISA, France}{ocan.sankur@irisa.fr}{0000-0001-8146-4429}{}%{email}{orcid}{funding}

\author{Christel Baier}{Technische Universit\"at Dresden, Germany}{Christel.Baier@tu-dresden.de}{0000-0002-5321-9343}{}%{email}{orcid}{funding}

\funding{This work was funded by DFG grant 389792660 as part of TRR 248, the Cluster of Excellence
EXC 2050/1 (CeTI, project ID 390696704, as part of Germany’s Excellence Strategy), and DFG-projects BA-1679/11-1 and BA-1679/12-1.}

\relatedversion{This is the extended version of the conference version accepted for publication at ICALP 2022.}

\authorrunning{ J. Piribauer,  O. Sankur, and C. Baier}
\Copyright{Christel Baier,  Jakob Piribauer, and Ocan Sankur}

\ccsdesc[500]{Theory of computation~Verification by model checking}
\keywords{Markov decision process, variance, stochastic shortest path problem}

\begin{document}

%%%%%%%%%%%%%%%%%%%%%%%%%%%%%%%%%%%%%%%%%%%%%%%%%%%%%%%%%%%%%%%%%%%%%%%%%%%%
%%%
%%%     title
%%%
%%%%%%%%%%%%%%%%%%%%%%%%%%%%%%%%%%%%%%%%%%%%%%%%%%%%%%%%%%%%%%%%%%%%%%%%%%%%

\title{The variance-penalized stochastic shortest path problem}

\maketitle

\begin{abstract}

The stochastic shortest path problem (SSPP) asks to resolve 
the non-deterministic choices in a Markov decision process (MDP) such that the expected accumulated weight  before reaching a target state is maximized.
This paper addresses the optimization of the variance-penalized expectation (VPE) of the accumulated weight, which is a variant of the SSPP in which a multiple of the  variance of accumulated weights is incurred as a penalty. It is shown that the optimal  VPE in MDPs with non-negative weights as well as an optimal deterministic finite-memory scheduler can be computed in exponential space.
The threshold problem whether the maximal VPE exceeds a given rational is shown to be EXPTIME-hard and to lie in NEXPTIME.
Furthermore, a result of interest in its own right obtained on the way is that a variance-minimal scheduler among all expectation-optimal schedulers can be computed in polynomial time.

\end{abstract}

%%%%%%%%%%%%%%%%%%%%%%%%%%%%%%%%%%%%%%%%%%%%%%%%%%%%%%%%%%%%%%%%%%%%%%

\section{Introduction}

Markov decision processes (MDPs) are a standard operational model comprising randomization and non-determinism and are widely used  in verification, articifical intelligence, robotics, and operations research. 
In each state of an MDP, there is a non-deterministic choice from a set of actions.  Each action is equipped with a weight and a probability distribution according to which the successor state is chosen randomly. 
In  the analysis of systems modelled as MDPs, one typically  is interested in the worst- or best-case behavior, where worst and best case range over all resolutions of the non-determinism.
So, the resulting algorithmic problems on MDPs  usually  ask to resolve non-deterministic choices by specifying a \emph{scheduler}  such that the resulting probabilistic behavior is optimized with respect to an objective function.
If the weights are used to model one of various quantitative aspects of a system such as 
costs, resource consumption, rewards, or utility, a frequently encountered such optimization problem is the 
 \emph{stochastic shortest path problem} (SSPP) \cite{BerTsi91,deAlf99}.
It asks to optimize the expected value of the accumulated weight before reaching a target state.
Example applications include the analysis of worst-case expected termination times of probabilistic programs or finding the optimal controls in a motion planning scenario with random external influences.

While a solution to the SSPP provides guarantees on the behavior of a system in all environments or indicates the optimal control to maximize expected rewards, it completely disregards all other aspects of the resulting probability distribution of the accumulated weight besides the expected value.
In almost all practical applications, however, the uncertainty coming with the probabilistic behavior cannot be neglected.
In traffic control systems or energy grids, for example, large variability in the throughput comes at a high cost due to the risk of traffic jams or the difficulty of storing surplus energy. Also a probabilistic program employed in a complex environment might be of more use with a higher  expected termination time in exchange for a lower chance of extreme termination times.

To overcome these shortcomings of the SSPP, various additional optimization problems have been studied in the literature:
Optimizing conditional expected accumulated weights under the condition that certain system states are reached allows for a more fine-grained system analysis by making it possible to determine the worst- or best-case expectation in different scenarios   \cite{tacas2017,fossacs2019}.
Given a probability $p$, quantiles on the accumulated weight in MDPs, also called \emph{values-at-risk} in the context of risk analysis, are the best bound $B$ such that the accumulated weight exceeds $B$ with probability at  most $p$ in the worst or best case \cite{HaaseKiefer15,UB13}. 
The \emph{conditional value-at-risk} and the \emph{entropic value-at-risk} are more involved measures that have been studied in this context \cite{ahmadi2021,kretinsky2018}. They quantify how far the probability mass of  the tail of the probability distribution lies above the {value-at-risk}.
The arguably most prominent measure for the deviation of a random variable from its expected value is the \emph{variance}. The computation of the variance of accumulated weights has been studied in Markov chains \cite{verhoeff2004reward} and in MDPs \cite{mandl1971variance,MannorTsitsiklis2011}. The investigations of variance in MDPs in the literature is discussed in more detail in the `Related Work' section below.

\tudparagraph{1ex}{Variance-penalized expectation (VPE).}
In this paper, we investigate a variant of the SSPP in which the costs caused by probabilistic uncertainty are priced in to the objective function:
We study the optimization of the \emph{variance-penalized expectation} (VPE), a well-known measure that combines the expected 
value $\mu$ and the variance $\sigma^2$ into the single objective function $\mu - \lambda \cdot \sigma^2$ where  $\lambda$ is a parameter that can be varied to aim for different tradeoffs between expectation and variance.
In the context of optimization problems on MDPs, the VPE has been studied, e.g., in \cite{filar1989variance,collins1997finite}.

Furthermore, the VPE   finds use in an area of research primarily concerned with the tradeoffs between expected performance and risks, namely, the theory of financial markets and investment decision-making:
In 1952, 
Harry Markowitz introduced \emph{modern portfolio theory}  that evaluates portfolios in terms
of expected returns and variance of the returns  \cite{markowitz52},  for which he was later awarded the Nobel Prize in economics. 
 A portfolio lies on the \emph{Markowitz efficient frontier} if the expected return cannot be increased without increasing the variance and, vice versa,  the variance cannot be decreased without decreasing the expectation. The final choice  of a portfolio on the efficient frontier depends on the investors preferences.
 In this context, the VPE $\mu - \lambda \cdot \sigma^2$ is a simple, frequently used way to express the preference of an investor using the single parameter $\lambda$ capturing the risk-aversion of the investor (see, e.g., \cite{goetzmann2014modern}).
 In more involved accounts, the investor's preference  is described in terms of a utility function mapping returns to utilities. For the commonly used exponential utility function $u(x)=-e^{-\alpha x}$ and normally distributed returns, the objective of an investor trying to maximize expected utility turns  out to be equivalent   to the maximization  of the VPE with parameter $\lambda=\alpha/2$ \cite{arrow70,pratt64}.

For an illustration of the VPE, consider the following example:

\begin{figure}[t]
  \begin{subfigure}[b]{.5\textwidth}
    \resizebox{1\textwidth}{!}{%
      \begin{tikzpicture}[scale=1,auto,node distance=8mm,>=latex]
        \tikzstyle{round}=[thick,draw=black,circle]

        \node[round, draw=black,minimum size=10mm] (sinit) {$\sinit$};
        \node[above=35mm of sinit] (input) {};
        \node[round,right=35mm of input,minimum size=10mm] (d) {$d$};
        \node[round, right=10mm of input, minimum size=10mm] (c) {$c$};
        \node[round, left=10mm of input, minimum size=10mm] (b) {$b$};
        \node[round,left=35mm of input,minimum size=10mm] (a) {$a$};
        
        \node[round,above=35mm of input,minimum size=10mm] (goal) {$\goal$};

        \draw[color=black ,->,very thick] (sinit) edge  node [pos=0.5,left] {$\alpha$} (a) ;
        \draw[color=black ,->,very thick] (sinit) edge  node [pos=0.5,left] {$\beta$} (b) ;
         \draw[color=black ,->,very thick] (sinit) edge  node [pos=0.5,right] {$\delta$} (d) ;
        \draw[color=black ,->,very thick] (sinit) edge  node [pos=0.5,right] {$\gamma$} (c) ;

        \draw[color=black ,->,very thick] (a) edge [bend left] node [pos=0.16,left] {$+0$} (goal) ;

        \draw[color=black ,->,very thick] (b)  edge  node [very near start, anchor=center] (m5) {} node [pos=0.2,right] {$2/3$} (goal) ;
        \draw[color=black ,->,very thick] (b) edge [loop, out=120, in=180, min distance=15mm]  node [very near start, anchor=center] (m6) {} node [pos=0.3,below] {$1/3$} (b) ;
        \draw[color=black , very thick] (m5.center) edge [bend right=35] node [pos=0.5,above=1pt] {$+1$} (m6.center);

        \draw[color=black ,->,very thick] (c)  edge  node [very near start, anchor=center] (h5) {} node [pos=0.2,left] {$9/10$} (goal) ;
        \draw[color=black ,->,very thick] (c) edge [loop, out=60, in=0, min distance=15mm]  node [very near start, anchor=center] (h6) {} node [pos=0.27,below=2pt] {$1/10$} (c) ;
        \draw[color=black , very thick] (h5.center) edge [bend left=35] node [pos=0.5,above=1pt] {$+3$} (h6.center);
        
        \draw[color=black ,->,very thick] (d)  edge [bend right] node [pos=.08, anchor=center] (l5) {} node [pos=0.16,left] {$3/4$} (goal) ;
        \draw[color=black ,->,very thick] (d) edge [loop, out=60, in=0, min distance=15mm]  node [very near start, anchor=center] (l6) {} node [pos=0.3,below] {$1/4$} (d) ;
        \draw[color=black , very thick] (l5.center) edge [bend left=35] node [pos=0.5,above=1pt] {$+3$} (l6.center);

      \end{tikzpicture}
    }
  \end{subfigure}
  \hspace{-12pt}
  \begin{subfigure}[b]{.5\textwidth}
    \begin{tikzpicture}[yscale=0.8,xscale=.8]
      \begin{axis}[
          axis lines = middle,
          xlabel={$\mu$},
          ylabel={$\sigma^2$},
          ymin=-1, ymax=6.5,
          xmin=-.2, xmax=5,
          xticklabels={0,1,2,3,4},
          yticklabels={0,1,2,3,4,5,6},
          xtick={0,1,2,3,4},
          ytick={0,1,2,3,4,5,6},
        ]
        \addplot [name path=A,domain=0:1.5,
          samples=300,
          color=black,thick]
        {0.5*x+x*(1.5-x)};
         \addplot [name path=D,domain=0:1.5,
          samples=300,
          color=black,thick]
        {x+x*(4-x)};
        \addplot [name path=B,domain=1.5:10/3,
          samples=300,
          color=black,thick]
        {.75+(x-1.5)*13/66+(x-1.5)*(10/3-x)};
         \addplot [name path=E,domain=1.5:10/3,
          samples=300,
          color=black,thick]
        {x+x*(4-x)};
        \addplot [name path=C,domain=10/3:4,
          samples=300,
          color=black, thick]
        {10/9+(x-10/3)*13/3+(x-10/3)*(4-x)};
        \addplot [name path=F,domain=10/3:4,
          samples=300,
          color=black,thick]
        {x+x*(4-x)};

        \addplot[black!8] fill between[of=A and D];
        \addplot[black!8] fill between[of=B and E];
        \addplot[black!8] fill between[of=C and F];

        \node[label={315:{$\alpha$}},circle,fill,inner sep=2pt] at (axis cs:0,0) {};
        \node[label={270:{$\beta$}},circle,fill,inner sep=2pt] at (axis cs:1.5,.75) {};
        \node[label={270:{$\gamma$}},circle,fill,inner sep=2pt] at (axis cs:10/3,10/9) {};
        \node[label={270:{$\delta$}},circle,fill,inner sep=2pt] at (axis cs:4,4) {};

        \draw[dashed,thick,color=blue] (axis cs:2.222,0) -- (axis cs:5,2.777);
      \draw[thick,color=blue,->] (axis cs:3.666,1.444) -- (axis cs:4,.5);
      % \draw[dashed,color=black] (axis cs:3.3333,1.1111) -- (axis cs:4,4);
      \node[label={0:{\textcolor{blue}{$\mu-1\cdot \sigma^2$}}}] at (axis cs:3.7,1.3) {};
      \end{axis}
    \end{tikzpicture}
  \end{subfigure}
  \caption{The left hand side shows the  MDP $\cM$ for Example \ref{ex:intro}. On the right hand side, all possible combinations of expected accumulated weight and variance for schedulers for $\cM$ are depicted. The points corresponding to the four deterministic schedulers are marked by the corresponding action. Furthermore, the blue line indicates  all points at which $\mu-1\cdot \sigma^2=20/9$ and the arrow indicates the direction in which the value of this objective function increases.}
  \label{fig:ex_intro}
\end{figure}
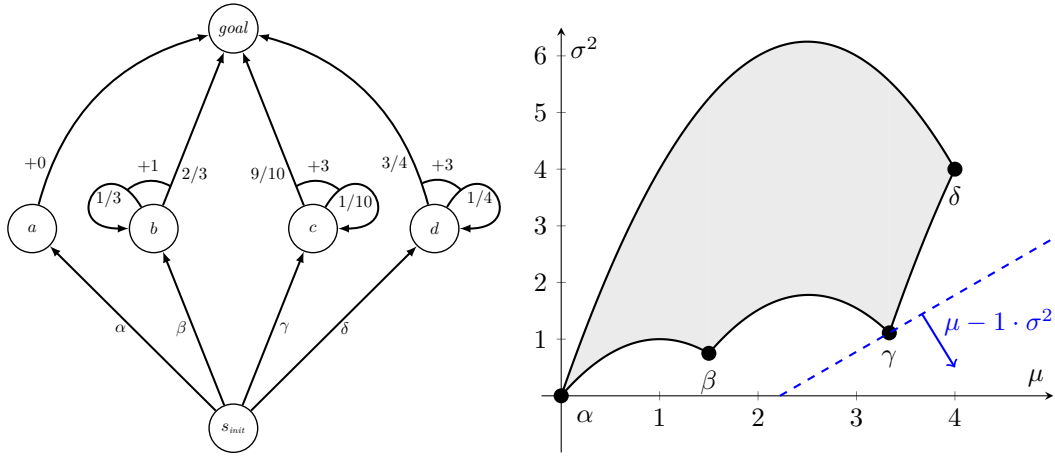

\begin{example}\label{ex:intro}
  Consider the MDP $\cM$ depicted in Figure \ref{fig:ex_intro} where non-trivial probability values as well as the weights accumulated are denoted next to the transitions. We want to analyze the possible trade-offs between the variance and the expected value of the accumulated weight that we can achieve in this MDP.

  The only non-deterministic choice is in the state $\sinit$. Choosing action $\alpha$ leads to $\goal$ with expected weight and variance $0$. For the remaining actions, the accumulated weight follows a geometric distribution where in each step some weight $k$ is accumulated and $\goal$ is reached with some probability $p$ after the step.
  For such a distribution, it is well-known that the expected accumulated weight is $k/p$ and the variance is $(k/p)^2\cdot (1-p)$.
  Plugging in the respective values for the distributions reached after actions $\beta$, $\gamma$, and $\delta$, we obtain the pairs of expectations and variances as depicted on the right-hand side of Figure \ref{fig:ex_intro}.
  In particular, choosing $\gamma$ leads to an expectation of $10/3$ and a variance of $10/9$.

 Making use of randomization over two different actions $\tau$ and $\sigma$ with probability $p$ and $1-p$, respectively, for some $p\in(0,1)$, we will see in Remark \ref{rem:parabolic} in Section \ref{sec:VPE} that the expected values and variances under the resulting schedulers lie on a parabolic line segment depicted in black that is uniquely determined by the expected values and variances  under $\tau$ and $\sigma$.  
 By further randomization over multiple actions, combinations of expectation and variance in the gray region 
 in Figure \ref{fig:ex_intro}  can be realized.
 
 Consider now the VPE with parameter $\lambda=1$. The dashed blue line in Figure \ref{fig:ex_intro}
 marks all points at which $\mu- 1\cdot \sigma^2=20/9$. The arrow indicates in which direction the value of the VPE increases.
 So, it turns out that choosing action $\gamma$ maximizes the VPE in this case; the slightly lower expectation compared to $\delta$ is compensated by a significantly lower variance.
 Geometrically, we can observe that the optimal point for the VPE for any parameter will always  lie on the border of the convex hull of the region of feasible points in the $\mu$-$\sigma^2$-plane as the VPE is a linear function of expectation and variance. 
 For varying values of $\lambda$,  also $\alpha$ (for $\lambda \geq 3$) and
  $\delta$ (for $\lambda\leq 1/13$) can constitute the optimal choice in $\sinit$ for the maximization of the VPE, while $\beta$ is not optimal for any choice of $\lambda$ as it lies in the interior of the convex hull of the feasible region.
  The results of  Section \ref{sec:VPE} will show that in general, the optimal point for the VPE can be achieved by a deterministic finite-memory scheduler.
  \Ende
\end{example}

%%%%%%%%%%%%%%%%%%%%%%%%%%%%%%

\tudparagraph{1ex}{Contribution.} 
The main results of this paper are the following:
\begin{enumerate}
\item
Among all schedulers that optimize the expected accumulated weight before reaching a target, a variance-minimal scheduler can be computed in polynomial time and chosen to be memoryless and deterministic (Section \ref{sec:minimal_variance}).
%The algorithm is based on a system of optimality equations that is also presented in \cite{mandl1971variance}. Our proof of the soundness of these equations, however, is much more elementary.
\item
The maximal VPE in MDPs with non-negative weights can be computed in exponential space.
The maximum is obtained by a deterministic scheduler that can be computed in exponential space as well (Section \ref{sec:VPE}).
As memory, an optimal scheduler only needs to keep track of the accumulated weight up to a bound computable in polynomial time. As soon as the bound is reached, optimal schedulers can switch to the behavior of a variance-minimal scheduler among the expectation-minimal schedulers  that can be computed by result 1.
\item
The threshold problem whether the maximal VPE is greater or equal to a rational $\vartheta$ is in NEXPTIME and EXPTIME-hard (Section \ref{sec:VPE}).
\end{enumerate}

\tudparagraph{1ex}{Related work.}
% Accumulated weights (this will perhaps go to the intro)
\indent\textit{Accumulated rewards.}
In \cite{mandl1971variance}, a characterization
of variance-minimal scheulders among the schedulers maximizing the expected accumulated weight in MDPs is given.
Here, we provide a simpler proof based on the calculations of \cite{verhoeff2004reward}; we moreover show
how to compute such schedulers in polynomial time. \cite{mandl1971variance} also contains hints for a similar
characterization of discounted reward, and developments for mean payoff.
Another closely related work is \cite{MannorTsitsiklis2011} which study the following multi-objective problem for the accumulated weight in finite-horizon MDPs:
given $\eta,\nu$ is there a scheduler achieving an expectation of at least~$\eta$, and a variance of at most~$\nu$?
This problem is shown to be NP-hard,
and exact pseudo-polynomial time algorithm is given for the existence of a scheduler with expectation $\eta$ and variance $\leq \nu$.
Furthermore, pseudo-polynomial approximation algorithms are given for optimizing the expectation under a constraint on the variance,
and optimizing the variance under a constraint on the expectation.

% Discounted accumulated weights
\textit{Discounted rewards.}
In \cite{Jaquette-1973}, the author proves that memoryless \emph{moment-optimal} schedulers exist for the discounted reward, that is, schedulers
that maximize the expectation, minimize the variance, maximize the third moment, and so on. Moreover, an algorithm is described
to compute such schedulers.
In \cite{sobel1982}, a formula for the variance of the discounted reward is given for memoryless schedulers and for the finite-horizon case,
in MDPs and semi-MDPs.
Variance-minimal schedulers among those maximizing the expected discounted reward until a target set is reached
are studied in \cite{wu_guo_2015} for MDPs with varying discount factors.
\cite{xia2018mean} presents a policy iteration algorithm to minimize variance of the discounted weight among schedulers achieving
an expectation equal to a given constant.
% ?
%\cite{xia2018mean} presents a necessary condition for schedulers optimizing the expected discounted sum, and minimizing variance,
%and defines an iterative procedure that reduces the variance to yield a local optimum.
%
%{\color{red}I couldn't have access to \cite{xia2018mean}. To be checked.}

% Mean payoff
\textit{Mean payoff.}
For mean payoff objectives, variance was studied in \cite{sobel1994} for memoryless strategies, and
algorithms were given to compute schedulers that achieve given bounds on the expectation and the variance~\cite{brazdil2017trading}.
The latter paper also considers the minimization of the variability, which is the average of the squared differences between the expected mean-payoff and each observed one-step reward.
In \cite{kurano1987markov}, the author considers optimizing the expected mean payoff and the average variance.
Average variance is defined as the limsup of the variances of the partial sums. They show how to minimize average variance among $\epsilon$-optimal strategies for the expected mean payoff.
Policy iteration algorithms were given in \cite{xia2016optimization,xia2018variance} to minimize variance or variability of the mean payoff (without  constraints on the expectation).

\textit{Variance-penalized expectation.}
The VPE was studied for finite-horizon  MDPs with terminal rewards in \cite{collins1997finite}.
In \cite{filar1989variance}, this notion was studied for the expectation and the variability of both mean payoff and discounted rewards.
\cite{xia2020risk} presents a policy iteration algorithm converging against \emph{local} optima for a similar measure.

\section{Preliminaries}\label{sec:prelim}

We give basic definitions and present our notation (for  details, see, e.g., \cite{Puterman}). Afterwards, we provide auxiliary results on expected frequencies  used in the subsequent sections.

\subsection{Notation and definitions}

\tudparagraph{.3ex}{Notations for Markov decision processes.}
A \emph{Markov decision process} (MDP) is a tuple $\mathcal{M} = (S,\Act,P,\sinit,\goal,\wgt)$
where $S$ is a finite set of states,
$\Act$ a finite set of actions,
$P \colon S \times \Act \times S \to [0,1] \cap \Rational$  the
transition probability function,
$\sinit \in S$ the initial state, $\goal\in S$ a designated target state,
and
$\wgt \colon S \times \Act \to \mathbb{Z}$ the weight function.
We require that
$\sum_{t\in S}P(s,\act,t) \in \{0,1\}$
for all $(s,\alpha)\in S\times \Act$.
We say that action $\alpha$ is enabled in state $s$ iff $\sum_{t\in S}P(s,\act,t) =1$ and denote the set of all actions that are enabled in state $s$ by $\Act(s)$.
In this paper, for all MDPs, we assume that $\goal$ is the only \emph{trap} state in which no actions are enabled,
that $\goal$ is reachable from all other states $s$,  and
that all states are reachable from $s_{\mathit{init}}$.
The paths of $\cM$ are finite or
infinite sequences $s_0 \, \act_0 \, s_1 \, \act_1  \ldots$
where states and actions alternate such that
$P(s_i,\act_i,s_{i+1}) >0$ for all $i\geq0$.
For $\fpath =
    s_0 \, \act_0 \, s_1 \, \act_1 \,  \ldots \act_{k-1} \, s_k$,
$\wgt(\fpath)=
    \wgt(s_0,\act_0) + \ldots + \wgt(s_{k-1},\act_{k-1})$
denotes the accumulated weight of $\pi$,
$P(\fpath) =
    P(s_0,\act_0,s_1)
    \cdot \ldots \cdot P(s_{k-1},\act_{k-1},s_k)$
its probability, and
$\last(\fpath)=s_k$ its last state. A path is called \emph{maximal} if it is infinite or ends in the trap state $\goal$.
The \emph{size} of $\cM$
is the sum of the number of states
plus the total sum of the logarithmic lengths of the non-zero
probability values
$P(s,\alpha,s')$ as fractions of co-prime integers and the weight values $\wgt(s,\alpha)$.

An \emph{end component} of $\cM$ is a strongly connected sub-MDP formalized by a subset $S^\prime\subseteq S$ of states and a non-empty subset $\mathfrak{A}(s)\subseteq \Act(s)$  for each state $s\in S^\prime$ such that for each $s\in S^\prime$, $t\in S$ and $\alpha\in \mathfrak{A}(s)$ with $P(s,\alpha,t)>0$, we have $t\in S^\prime$ and such that in the resulting sub-MDP all states are reachable from each other.
An end-component is a $0$-end-component if
it only contains cycles whose accumulated weight is $0$ (so-called $0$-cycles) so that the accumulated weight is bounded on all (infinite) paths in the end component.
We will further use the \emph{mean payoff} measure as tool to classify end-components.
For an infinite path $\zeta$, the mean payoff is defined as $\MeanPayoff(\zeta) = \lim\inf_{n\rightarrow \infty} \frac{1}{n} \wgt(\prefix{\zeta}{n})$ where $\prefix{\zeta}{n}$ is the prefix of length $n$ of $\zeta$.

%%%%%%%%%%%%%%%%%%%%%%%%%%%%%%%%%%%%%%%%%%%%

\tudparagraph{.3ex}{Scheduler.}
A \emph{scheduler} for $\cM$
is a function $\sched$ that assigns to each non-maximal path $\fpath$
a probability distribution over $\Act(\last(\fpath))$.
If the choice of a scheduler $\sched$ depends only on the current state, i.e., if $\sched(\fpath)=\sched(\fpath^\prime)$ for all non-maximal paths $\fpath$ and $\fpath^\prime$ with $\last(\fpath)=\last(\fpath^\prime)$,
we say that $\sched$ is \emph{memoryless}. In this case, we also view schedulers as functions mapping states $s\in S$ to probability distributions over $\Act(s)$.
A scheduler $\sched$ that satisfies $\sched(\fpath)=\sched(\fpath^\prime)$ for all pairs of finite paths $\fpath$ and $\fpath^\prime$ with $\last(\fpath)=\last(\fpath^\prime)$ and $\wgt(\fpath)=\wgt(\fpath^\prime)$ is called \emph{weight-based} and can be viewed as a function from state-weight pairs $S\times \mathbb{Z}$ to probability distributions over actions.
If there is a finite set $X$ of memory modes and a memory update function $U:S\times \Act \times S \times X \to X$ such that the choice of $\sched$ only depends on the current state after a finite path and the memory mode obtained from updating the memory mode according to $U$ in each step, we say that $\sched$ is a finite-memory scheduler.
A scheduler $\sched$ is called deterministic if $\sched(\fpath)$ is a Dirac distribution
for each path $\fpath$ in which case we also view the scheduler as a mapping to actions in $\Act(\last(\fpath))$.
Given a scheduler $\sched$,
$\zeta \, = \, s_0 \, \act_0 \, s_1 \, \act_1 \ldots$
is a $\sched$-path iff $\zeta$ is a path and
$\sched(s_0 \, \act_0  \ldots \act_{k-1} \, s_k)(\act_k)>0$
for all $k \geq 0$.
Given a scheduler $\sched$ and a finite $\sched$-path $\pi$, we define the residual scheduler $\after{\sched}{\pi}$ by
$
    \after{\sched}{\pi} (\rho) = \sched (\pi\circ \rho)
$
for each finite path $\rho$ starting in $\last(\pi)$.

\tudparagraph{.3ex}{Probability measure.}
We write $\Pr^{\sched}_{\cM,s}$ 
to denote the probability measure induced by a scheduler $\sched$ and a state $s$ of an MDP $\cM$.
It is defined on the $\sigma$-algebra generated by the {cylinder sets} $\Cyl(\pi)$  of all maximal extensions of a finite path  $\pi =
    s_0 \, \act_0 \, s_1 \, \act_1 \,  \ldots \act_{k-1} \, s_k$ starting in state $s$, i.e., $s_0=s$, by assigning  to $\Cyl(\pi)$ the probability that $\pi$ is realized under $\sched$, which is
   $ \sched(s_0)(\act_0) \cdot P(s_0,\act_0,s_1) \cdot \sched(s_0\act_0s_1)(\act_1)
    \cdot \ldots \cdot \sched(s_0\act_0 \dots s_{k-1})(\act_{k-1}) \cdot P(s_{k-1},\act_{k-1},s_k)$.
For details, see \cite{Puterman}.

For a random variable $X$ that is defined on (some) maximal paths in $\cM$, we denote the expected value of $X$ under the probability measure induced by a scheduler $\sched$ and state $s$ by $\mathbb{E}^{\sched}_{\cM,s}(X)$.
We define
$\mathbb{E}^{\min}_{\cM,s}(X) = \inf_{\sched} \mathbb{E}^{\sched}_{\cM,s}(X)$
and
$\mathbb{E}^{\max}_{\cM,s}(X) = \sup_{\sched} \mathbb{E}^{\sched}_{\cM,s}(X)$
where $\sched$ ranges over all schedulers for $\cM$ under which $X$ is defined almost surely.
The variance of $X$ under the probability measure determined by $\sched$ and $s$ in $\cM$ is denoted by $\Var^{\sched}_{\cM,s}(X)$ and defined by
\[
    \Var^{\sched}_{\cM,s}(X)\eqdef\mathbb{E}^{\sched}_{\cM,s}((X-\mathbb{E}^{\sched}_{\cM,s}(X))^2)=\mathbb{E}^{\sched}_{\cM,s}(X^2) -\mathbb{E}^{\sched}_{\cM,s}(X)^2.
\]
%It is well-known that $\Var^{\sched}_{\cM,s}(X)=\mathbb{E}^{\sched}_{\cM,s}(X^2) -\mathbb{E}^{\sched}_{\cM,s}(X)^2$.
Furthermore, for a measurable set of paths $\psi$ with positive probability, $\mathbb{E}^{\sched}_{\cM,s}(X|\psi)$ denotes the conditional expectation of $X$ under $\psi$.
If $s=\sinit$, we sometimes drop the subscript $s$.

These notations are extended to end-components of a given MDP, which are themselves seen as MDPs.
We may, for instance, write $\mathbb{E}^{\min}_{\cE,s}(X)$ where~$\cE$ is an end-component of~$\cM$,
and~$s$ is a state in~$\cE$, and the minimization ranges over schedulers of~$\cM$ that do not leave~$\cE$.

\tudparagraph{1ex}{Accumulated weight.}
For maximal paths $\zeta$ of $\cM$, we define the following random variable $\rawdiaplus \goal$:
\[
    \rawdiaplus\goal(\zeta)=
    \begin{cases}
        \wgt(\zeta)        & \text{ if }\zeta \vDash \Diamond \goal, \\
        \mathit{undefined} & \text{otherwise}.
    \end{cases}
\]
Recall that we only take schedulers under which a random variable is defined almost surely into account when addressing minimal or maximal expected values. For the expected value of $\rawdiaplus\goal$ to be defined, it is necessary that $\goal$ is reached almost surely.
We call a scheduler $\sched$ with $\Pr^{\sched}_{\cM}(\lozenge \goal)=1$ \emph{proper}. So, in the definition of the maximal (or minimal) expected accumulated weight
$\mathbb{E}^{\max}_{\cM}(\rawdiaplus \goal)=\sup_{\sched}\mathbb{E}^{\sched}_{\cM}(\rawdiaplus \goal)$, $\sched$ ranges over all proper schedulers.

\subsection{Auxiliary conclusions from results on expected frequencies}

In this section, we present conclusions from well-known results on the expected frequencies of state-weight pairs in MDPs in the formulation in which we use them in the paper. 
Let $\cM=(S,\Act, P, \sinit,\goal, \wgt)$ be an MDP with weights in $\mathbb{Z}$ and let $\sched$ be a scheduler.
For each state-weight pair $(s,w)\in S\times\mathbb{Z}$, we define the \emph{expected frequency} $\vartheta_{s,w}^{\sched}$ under $\sched$ by
\[
    \vartheta_{s,w}^{\sched}\eqdef \mathbb{E}^{\sched}_{\cM} (\text{number of visits to $s$ with accumulated weight $w$})
\]
where the random variable  ``number of visits to $s$ with accumulated weight $w$'' counts the number of prefixes $\pi$ of a maximal paths $\zeta$
with $\last(\pi)=s$ and $\wgt(\pi)=w$. Note also that in MDPs $\cM$ in which all end components have negative maximal expected mean-payoff, the expected frequencies of all state-weight pairs are finite under any scheduler.

\begin{restatable}{lemma}{lemmafrequencyweightbased}
    Let $\cM$ be an MDP and let $\sched$ be a scheduler such that the expected frequency $\vartheta_{s,w}^{\sched}$ are finite for all state-weight pairs $(s,w)\in S\times\mathbb{Z}$. Then, there is a weight-based (randomized) scheduler $\tsched$ with
    $
        \vartheta_{s,w}^{\sched}=\vartheta_{s,w}^{\tsched}
    $
    for all  $(s,w)\in S\times\mathbb{Z}$.
\end{restatable}

\begin{proof}[Proof sketch]
    Analogous to \cite[Theorem 5.5.1]{Puterman}: For each state-weight pair $(s,w)$ and each action $\alpha\in \Act(s)$, let $\vartheta_{s,w,\alpha}^{\sched}$ be the expected number of times that $\alpha$ is chosen under $\sched$ after   finite path ending in state $s$ with weight $w$.
    Define the scheduler $\tsched$ as a function from $S\times \mathbb{Z} \to \mathrm{Distr}(\Act)$ by letting
    \[
        \tsched(s,w)(\alpha) \eqdef \frac{\vartheta_{s,w,\alpha}^{\sched}}{\vartheta_{s,w}^{\sched}}.\qedhere
    \]
\end{proof}

\begin{corollary}\label{cor:weight-based}
    Let $\cM$ be an MDP. Let $\sched$ be a scheduler for which $\mathbb{E}^{\sched}_{\cM} (\rawdiaplus \goal)$ and $\Var^{\sched}_{\cM} (\rawdiaplus \goal)$ are defined and for which the expected frequency $\vartheta_{s,w}^{\sched}$ are finite for all state-weight pairs $(s,w)\in S\times\mathbb{Z}$. Then, there is a weight-based scheduler $\tsched$ with
    \[
        \mathbb{E}^{\sched}_{\cM} (\rawdiaplus \goal) = \mathbb{E}^{\tsched}_{\cM} (\rawdiaplus \goal) \qquad \text{ and } \qquad \Var^{\sched}_{\cM} (\rawdiaplus \goal) = \Var^{\tsched}_{\cM} (\rawdiaplus \goal).
    \]
\end{corollary}

\begin{proof}
    The expected value and the variance of $\rawdiaplus \goal$ under a scheduler $\sched$ depend only on the expected frequencies $\vartheta^{\sched}_{\goal,w}$ with $w\in \mathbb{Z}$.
\end{proof}

In this paper, we address questions concerning the possible combinations of expected value and variance of the random variable $\rawdiaplus\goal$. Due to this corollary, we can restrict our attention to weight-based schedulers for all investigations in the sequel.

Given two scheduler $\sched$ and $\tsched$, our definition of schedulers does not directly allow us to define a new scheduler $\rsched$ that behaves according to $\sched$ with probability $p\in (0,1)$ and according to $\tsched$ with probability $1-p$. For each state-weight pair $(s,w)$ the expected frequency under the hypothetical scheduler $\rsched$ would be $p\cdot \vartheta^{\sched}_{s,w}+(1-p)\cdot \vartheta^{\tsched}_{s,w}$. The following lemma states that a scheduler achieving these frequencies exists:

\begin{restatable}{lemma}{lemmafrequencyconvex}
    Let $\cM$ be an MDP as above and let $\sched$ and $\tsched$ be  schedulers such that the expected frequency $\vartheta_{s,w}^{\sched}$ and $\vartheta_{s,w}^{\tsched}$ are finite for all state-weight pairs $(s,w)\in S\times\mathbb{Z}$. Further, let $p\in (0,1)$. Then, there exists a scheduler $\rsched$ such that
    $
        \vartheta_{s,w}^{\rsched}=p\cdot \vartheta^{\sched}_{s,w}+(1-p)\cdot \vartheta^{\tsched}_{s,w}
    $
    for all state-weight pairs $(s,w)$.
\end{restatable}

\begin{proof}[Proof sketch]
Let $\vartheta_{s,w,\alpha}^{\sched}$ be defined as in the proof above.
    We define the weight-based scheduler $\rsched$ as follows: For all state-weight pairs $(s,w)$ and all $\alpha\in \Act(s)$, let
    \[
        \rsched(s,w)(\alpha)=\frac{p\cdot \vartheta^{\sched}_{s,w,\alpha}+(1-p)\cdot \vartheta^{\tsched}_{s,w,\alpha}}{p\cdot \vartheta^{\sched}_{s,w}+(1-p)\cdot \vartheta^{\tsched}_{s,w}}. 
    \]
    The proof of the correctness is analogous to \cite[Theorem 9.12]{Kallenberg}.
\end{proof}

This lemma allows us to introduce the following notation:

\begin{definition}\label{def:convex_scheduler}
    Given $\cM$, $\sched$ and $\tsched$ as in the previous lemma, we denote the scheduler $\rsched$ whose existence is stated in the lemma by $p\cdot \sched \oplus (1-p)\cdot \tsched$.
\end{definition}

\def\MD{\mathsf{MD}}
\section{Minimal variance among expectation-optimal schedulers}
\label{sec:minimal_variance}

Let us call a scheduler \emph{expectation-optimal} if it maximizies the expectation
of~$\rawdiaplus\goal$ from a given state~$s$. 
In this section, we prove a result that is of interest in its own right and that will play a crucial role in our investigation of the optimization of the VPE in the following section.
Namely, we show how to compute a scheduler that minimizes the variance
among expectation-optimal schedulers in polynomial time.
Note that in MDPs with weights in $\mathbb{Z}$, the mimimization of the expectation of $\rawdiaplus\goal$
can be reduced to the maximization by multiplying all weights with $-1$. This change of weights does not affect the variance and hence all results of this section also apply to expectation-minimal schedulers.

We  assume that in a given MDP $\cM=(S,\Act,P,\sinit,\wgt,\goal)$, the maximal achievable expectation of $\rawdiaplus\goal$ is finite.
This can be checked in polynomial time~\cite{BaiBerDubGbuSan-LICS18} and, when this value is finite, it is achievable
by memoryless deterministic strategies. 
By~\cite{BaiBerDubGbuSan-LICS18}, all end components~$E$ of $\cM$  are then either $0$-end components or satisfy $\mathbb{E}^{\max}_{E}(\MeanPayoff)<0$.

The algorithm proceeds as follows. First, a transformation is applied so as to ensure that
the only end-components in~$\cM$ are such that the maximal achievable expected mean payoff
is negative; while preserving the expectation and the variance of $\rawdiaplus\goal$ (Lemma~\ref{lemma:spider}).
We then prune the MDP so that all actions are optimal for maximizing the expected $\rawdiaplus\goal$
(Lemma~\ref{lemma:pruned-cM'}). It follows that all schedulers then achieve the same expected $\rawdiaplus\goal$.
We then derive an equation system in which the variances at each state are unknowns, while the expectations
are known constants (Lemma~\ref{lemma:variance-eqn-system}). We conclude by showing that this equation system admits a unique solution
and is solvable in polynomial time.
Omitted proofs can be found in Appendix \ref{app:minimal_variance}.

\begin{restatable}[\cite{BaiBerDubGbuSan-LICS18}]{lemma}{lemmapreprocessing}
    \label{lemma:spider}
    Let $\cM=(S,\Act,P,\sinit,\wgt,\goal)$ be an MDP with $\mathbb{E}^{\max}_{\cM}(\rawdiaplus \goal)<\infty$.
    There is a polynomial transformation which outputs an MDP~$\cM'$
    with the following properties:
    \begin{enumerate}
        \item $\cM'$ has no $0$-end-components,
        \item there is a mapping~$f$ from schedulers
              of~$\cM$ to those of~$\cM'$ such that for all \emph{proper} schedulers $\sched$ for~$\cM$,
              $\mathbb{E}^{\sched}_{\cM}(\rawdiaplus \goal) = \mathbb{E}^{f(\sched)}_{\cM'}(\rawdiaplus \goal)$,
              and $\mathbb{V}^{\sched}_{\cM}(\rawdiaplus \goal) = \mathbb{V}^{f(\sched)}_{\cM'}(\rawdiaplus \goal)$.
        \item there is a mapping~$g$ from schedulers
              of~$\cM'$ to those of~$\cM$ such that for all \emph{proper} schedulers $\sched$ for~$\cM'$,
              $\mathbb{E}^{g(\sched)}_{\cM}(\rawdiaplus \goal) = \mathbb{E}^{\sched}_{\cM'}(\rawdiaplus \goal)$,
              and $\mathbb{V}^{g(\sched)}_{\cM}(\rawdiaplus \goal) = \mathbb{V}^{\sched}_{\cM'}(\rawdiaplus \goal)$.
    \end{enumerate}
\end{restatable}

\medskip 

From now on, by the previous lemma, we assume that $\cM$ only has end-components~$E$
with $\mathbb{E}^{\max}_{\cM}(\MeanPayoff)<0$.
We start by computing $\mathbb{E}^{\max}_{\cM}(\rawdiaplus \goal)$ with the following equation:
\begin{equation*}
    \label{eqn:mu}
    \mu_s = \left\{\begin{array}{ll}
        0                                                                  & \text{ if } s = \goal, \\
        \max_{a \in \Act(s)}\sum_{s' \in S}P(s,a,s')(\wgt(s,a) + \mu_{s'}) & \text{ otherwise.}
    \end{array}\right. \tag{$\ast$}
\end{equation*}

By \cite{BerTsi91}, \eqref{eqn:mu} has the unique solution $\mu_s = \mathbb{E}^{\max}_{\cM}(\rawdiaplus \goal)$ and this solution is computable in polynomial time via linear programming.
Let us define $\Act^{\max}(s)$ as the set of actions from~$s$ which satisfy~\eqref{eqn:mu} with equality,
i.e. $\Act^{\max}(s) \eqdef \{ a \in \Act(s) \mid \mu_s = \wgt(s,a)+\sum_{s' \in S}P(s,a,s') \mu_{s'}\}$,
and let~$\cM'$ be obtained by restricting~$\cM$ to actions from~$\Act^{\max}$.
By standard arguments (see Appendix \ref{app:minimal_variance}), we can show the following lemma:

\begin{restatable}{lemma}{lemmapruned}
    \label{lemma:pruned-cM'}
    Let~$(\mu_s)_{s \in S}$ be the solution of~\eqref{eqn:mu} for an MDP~$\cM$.
    Let~$\cM'$ obtained from~$\cM$ as above.
    Then, $\cM'$ has no end-components. Moreover, for all $s \in S$,
    all schedulers~$\sched$ of $\cM'$ achieve $\mathbb{E}^\sched_{\cM'}[\rawdiaplus\goal] = \mu_s$.
\end{restatable}

\medskip

So, in order to find the variance-minimal scheduler among expectation optimal schedulers for $\cM$, it is sufficient to find a variance-minimal scheduler for $\cM'$. 
We derive the following lemma by adapting \cite{verhoeff2004reward} to MDPs.

\begin{restatable}{lemma}{lemmavarianceequation}
    \label{lemma:variance-eqn-system}
    Consider an MDP~$\cM$, and assume that there is a vector $(\mu_s)_{s \in S}$ of values
    such that all schedulers~$\sched$ satisfy
    $\forall s \in S, \mathbb{E}_{\cM,s}^{\sched}(\rawdiaplus\goal) = \mu_s$.
    Then, $(\mathbb{V}^{\inf}_{\cM,s}(\rawdiaplus\goal))_{s \in S}$ is the unique solution of the following equation:
    \begin{equation*}\label{eqn:variance}
        V_s = \left\{\begin{array}{ll}
            0                                                                                           & \text{ if } s = \goal, \\
            \min_{a \in \Act(s)} \sum_{t \in S}P(s,a,t)\left((\wgt(s,a) + \mu_t - \mu_s)^2 + V_t\right) & \text{ otherwise.}
        \end{array}\right. \tag{$\ast\ast$}
    \end{equation*}
\end{restatable}

Note that the equation system (\ref{eqn:variance}) is the same as the equation system used to minimize the expected accumulated weight before reaching 
$\goal$ under the weight function $\wgt'$ that assigns the non-negative weight $(\wgt(s,a) + \mu_t - \mu_s)^2$ to the transition $(s,\alpha,t)$.
So, this equation system is solvable in polynomial time \cite{deAlf99}. Using that all schedulers in $\cM'$ achieve an expected accumulated weight of $\mu_s$ when starting in state $s$,
the results of this section can be combined to the following theorem.
\begin{theorem}
    Given an MDP~$\cM$  such that $\mathbb{E}^{\max}_{\cM}[\rawdiaplus\goal] < \infty$,
     a memoryless deterministic, expectation-optimal scheduler~$\sched$ such that
    $\mathbb{V}_{\cM,s}^{\sched}[\rawdiaplus\goal]$  is minimal  among all expectation-optimal schedulers for any state $s$ is computable in polynomial time.
\end{theorem}

\section{Variance-penalized expectation}\label{sec:VPE}
The goal of this section is to develop an algorithm to compute the optimal  \emph{variance-penalized expectation} (VPE).
Given a rational $\lambda>0$, we define the VPE with parameter $\lambda$ under a scheduler $\sched$
as
\[
	\VP[\lambda]^{\sched}_{\cM} \eqdef \mathbb{E}^{\sched}_{\cM}(\rawdiaplus \goal) - \lambda \cdot \Var^{\sched}_{\cM} (\rawdiaplus \goal)
	= \mathbb{E}^{\sched}_{\cM}(\rawdiaplus \goal) - \lambda \cdot \mathbb{E}^{\sched}_{\cM}(\rawdiaplus \goal^2) + \lambda\cdot (\mathbb{E}^{\sched}_{\cM}(\rawdiaplus \goal))^2.
\]

\tudparagraph{1ex}{Task: }
Compute the maximal variance-penalized expectation
\[
	\VP[\lambda]_{\cM}^{\max}\eqdef \sup_{\sched} \VP[\lambda]^{\sched}_{\cM}
\]
where the supremum ranges over all proper schedulers.
Furthermore, compute an optimal  scheduler $\sched$ with $\VP[\lambda]_{\cM}^{\sched}=\VP[\lambda]_{\cM}^{\max}$.
\vspace{12pt}

Throughout this section,
we  will restrict ourselves to MDPs $\cM=(S,\Act,P,\sinit,\wgt,\goal)$  with a weight function $\wgt \colon S\times \Act \to \mathbb{N}$, i.e., we only consider MDPs with non-negative weights.
Key results established in this section do not hold in the general setting with arbitrary weights and further complications arise. In the conclusions we will briefly discuss these complications.

As before, we  are only interested in schedulers that reach the goal with probability $1$. If the maximal expectation $\mathbb{E}^{\max}_{\cM}(\rawdiaplus \goal)<\infty$, it is well-known that in this case of non-negative weights, all end components of $\cM$ are $0$-end components \cite{deAlf99,BaiBerDubGbuSan-LICS18}.
Hence, w.l.o.g., we can assume  that $\cM$ has no end components throughout this section by Lemma \ref{lemma:spider}.
In this case, $\rawdiaplus \goal$ is defined on almost all paths under any scheduler.
So, in particular the values $\mathbb{E}^{\sched}_{\cM}(\rawdiaplus \goal)$ and $\Var^{\sched}_{\cM}(\rawdiaplus \goal)$ are defined for all schedulers $\sched$. Furthermore, as we have seen in Corollary \ref{cor:weight-based}, it is sufficient to consider weight-based schedulers for the optimization of VPEs.
The main result of this section is the following:

\begin{theorem*}
Given an MDP $\cM$ and $\lambda$ as above, the optimal value $\VP[\lambda]_{\cM}^{\max}$ and an optimal scheduler $\sched$ can be computed 
in exponential space. Given a rational $\vartheta$, the threshold problem whether $\VP[\lambda]_{\cM}^{\max}\geq \vartheta$ is in NEXPTIME and EXPTIME-hard.
\end{theorem*}

To obtain the main result, we will first prove that the maximal VPE is obtained by a deterministic scheduler (Section \ref{sub:deterministic}).
This result can then  be used for the EXPTIME-hardness proof for the threshold problem (Section \ref{sub:hardness}).
The key step to obtain the upper bounds of the main result is to show that optimal schedulers have to \emph{minimize} the weight that is expected to still be accumulated after a computable bound of accumulated weight has been exceeded. We call such a bound a \emph{saturation point} (Section \ref{sub:saturation}). Finally, we show how to utilize the saturation point result to solve the threshold problem and to compute the optimal VPE (Section \ref{sub:computation}). Proofs omitted in this section can be found in Appendix \ref{app:VPE}.

\begin{remark}\label{rem:minimal_expectation_VPE}
In the formulation presented here, the goal is to maximize the expected accumulated weight with a penalty for the variance.
All results and proofs in this section, however, hold analogously for the variant $\sup_{\sched} - \mathbb{E}^{\sched}_{\cM}(\rawdiaplus \goal) - \lambda \cdot \Var^{\sched}_{\cM} (\rawdiaplus \goal)$ of the maximal VPE in which the goal is to minimize the expected accumulated weight while receiving  a penalty for the variance. In particular, the same saturation point works and optimal schedulers still have to minimize the expected accumulated weight as soon as the accumulated weight exceeds the saturation point.
\Ende
\end{remark}

\subsection{Existence of optimal deterministic schedulers} \label{sub:deterministic}

We begin this section with a lemma describing how  the variance of accumulated weight behaves under convex combinations of schedulers.
This will allow us to show that the maximal VPE can be approximated by deterministic schedulers with the help of Lemma \ref{lem:ranomization_convex} describing a connection between randomization and convex combinations.
This first lemma follows via basic arithmetic form the fact that the expected values of $\rawdiaplus \goal$ and $\rawdiaplus \goal^2$ depend linearly on the expected frequencies of the state-weight pairs $(\goal,w)$ with $w\in\mathbb{N}$.

\begin{restatable}{lemma}{lemparabolic}\label{lem:parabolic}
	Let $\cM=(S,\Act,P,\sinit,\wgt,\goal)$ be an MDP with non-negative weights and no end components. Let $\sched$ and $\tsched$ be two schedulers for $\cM$. Let $p\in (0,1)$.
	The scheduler $\rsched\eqdef p\cdot\sched\oplus(1-p)\cdot\tsched$  satisfies
	\[
		\Var^{\rsched}_{\cM}(\rawdiaplus\goal) = p\cdot \Var^{\sched}_{\cM}(\rawdiaplus\goal)+(1-p)\cdot \Var^{\tsched}_{\cM}(\rawdiaplus\goal) + p\cdot (1-p)\cdot (\mathbb{E}^{\sched}_{\cM}(\rawdiaplus\goal) - \mathbb{E}^{\tsched}_{\cM}(\rawdiaplus\goal))^2.
	\]
\end{restatable}

\begin{proof}[Proof sketch]
The claim follows from straight-forward calculations using that the expected values of $\rawdiaplus\goal$ and $\rawdiaplus\goal^2$ depend linearly on the expected frequencies of the state-weight pairs $(\goal,w)$ with $w\in\mathbb{N}$.
\end{proof}

\begin{remark}\label{rem:parabolic}
	Given two schedulers $\sched$ and $\sched^\prime$ under which the expectation and variance are $(\eta,\nu)$ and $(\eta^\prime,\nu^\prime)$, respectively, such that $\eta<\eta^\prime$,
	there is a unique convex combination $\tsched$ of the two schedulers with expectation $x$ for all $x\in [\eta,\eta^\prime]$.
	Viewing the variance of these convex combinations as a function $V:[\eta,\eta^\prime]\to \mathbb{R}$, we can observe the following using the previous Lemma \ref{lem:parabolic}:
	\[
		V(x)= \nu+\frac{x-\eta}{\eta^\prime-\eta}\cdot(\nu^\prime-\nu) + \frac{x-\eta}{\eta^\prime-\eta}\cdot \frac{\eta^\prime-x}{\eta^\prime-\eta} \cdot (\eta^\prime-\eta)^2 = \nu+\frac{x-\eta}{\eta^\prime-\eta}\cdot(\nu^\prime-\nu)  + (x-\eta)\cdot (\eta^\prime - x).
	\]
	The coefficient before $x^2$ in this quadratic polynomial hence is always $-1$.
	\Ende
\end{remark}

The following lemma stating the continuity of the VPE will be useful in several ways: If we manipulate schedulers at one state-weight pair at a time, we can reason about the scheduler we obtain in the limit after manipulating the scheduler at all state-weight pairs, e.g., in the proof of Theorem \ref{thm:determinist_scheduler} below. Further, it will allow us to prove that   there is  an optimal scheduler, i.e., that the supremum in the definition of 
$\VPE[\lambda]^{\max}_{\cM}$ is in fact a maximum.

\begin{restatable}[Continuity of VPE]{lemma}{lemmacontinuity}
	\label{lem:continuous}
	Let $\cM$ and $\lambda>0$ be as above.
	The variance-penalized expectation as a function from weight-based schedulers to $\mathbb{R}$ is (uniformly) continuous in the following sense:
	Given  $\varepsilon>0$, there is a natural number $N_\varepsilon$ such that for all weight-based schedulers $\sched$ and $\tsched$ that agree on all state-weight pairs $(s,w)$ with $w\leq N_\varepsilon$, we have
	\[
		\Big| \VPE[\lambda]^{\sched}_{\cM} - \VPE[\lambda]^{\tsched}_{\cM} \Big| < \varepsilon.
	\]
\end{restatable}

\begin{proof}[Proof sketch]
The claim follows from the fact that the probability that a high amount of weight $w$ is accumulated under any scheduler decreases exponentially as $w$ tends to $\infty$. 
\end{proof}

For the final ingredient to show that deterministic schedulers approximate the optimal VPE, we take a closer look at the relation of randomization to convex combinations of schedulers. For the following lemma,  let $\sched$ be a weight-based scheduler for an MDP $\cM$ as before. Assume that there is a state-weight pair $(s,w)\in S\times \mathbb{N}$ reachable under $\sched$ such that $\sched$ chooses two different actions $\alpha$ and $\beta$ with probabilities $q$ and $1-q$, respectively, for some $q\in (0,1)$.
    Let $\sched_\alpha$ be the scheduler that agrees with $\sched$ on all state-weight pairs except for $(s,w)$ and that chooses $\alpha$ with probability $1$ at $(s,w)$. Define $\sched_\beta$ analogously. 
    The technical proof of the following lemma can be found  in Appendix \ref{app:VPE}.

\begin{restatable}{lemma}{lemmaRandomizationConvex}\label{lem:ranomization_convex}
    Let $\cM$, $\sched$, $\sched_\alpha$, $\sched_\beta$, and $q$ be as above.
    There is a value $p\in(0,1)$ such that the expected frequencies of all state-weight pairs are the same under $\sched$ and
    $p\cdot \sched_\alpha \oplus (1-p)\cdot \sched_{\beta}$.
\end{restatable}

\begin{restatable}[Deterministic schedulers approximate optimal VPE]{theorem}{theoremsupremumdeterministic}\label{thm:determinist_scheduler}
	Let $\cM$ be an MDP with non-negative weights and without end components and let $\lambda>0$. For each scheduler $\sched$, there is a deterministic weight-based scheduler $\tsched$ with
	\[
		\VPE[\lambda]^{\tsched}_{\cM} \geq \VPE[\lambda]^{\sched}_{\cM}.
	\]
\end{restatable}

\begin{proof}[Proof sketch]
W.l.o.g., we can assume that $\sched$ is weight-based by Corollary \ref{cor:weight-based}. 
At a single state-weight pair $(s,w)$ at which $\sched$ makes use of randomization between, we can (potentially repeatedly) apply Lemma \ref{lem:ranomization_convex} and Lemma \ref{lem:parabolic} to find a scheduler $\sched^\prime$ that does not make use of this randomization but satisfies $\VPE[\lambda]^{\sched^\prime}_{\cM} \geq \VPE[\lambda]^{\sched}_{\cM}$. Going through all state-weight pairs in this fashion, we can construct an infinite sequence of schedulers with non-decreasing VPE in which randomization is successively removed at all state-weight pairs. In the limit, we obtain a well defined deterministic weight-based scheduler $\tsched$.  Lemma \ref{lem:continuous} allows us to conclude that $\VPE[\lambda]^{\tsched}_{\cM} \geq \VPE[\lambda]^{\sched}_{\cM}$.
\end{proof}

\noindent
In the definition of the maximal variance-penalized expectation
$
	\VPE[\lambda]^{\max}_{\cM} = \sup_{\sched}\VPE[\lambda]^{\sched}_{\cM}
$,
it is  sufficient to let the supremum range over deterministic weight-based schedulers $\sched$ in the light of this theorem.
For the proof of the existence of optimal schedulers, we make use of an analytic argument: Continuous functions on compact space obtain their maximum.
The continuity shown in Lemma \ref{lem:continuous} applied to the space of  deterministic  weight-based schedulers can be reformulated as continuity with respect to a metric on this space.
	Namely, we define the metric $d_{\cM}$ on the set of deterministic weight-based schedulers as follows: Given two deterministic weight-based schedulers $\sched$ and $\tsched$ for $\cM$, first let
	\[
		m(\sched,\tsched) \eqdef \min \{w \mid \text{there is a state $s\in S$ with } \sched(s,w)\not=\tsched(s,w)\}.
	\]
	We then define $d_{\cM}(\sched,\tsched) \eqdef 2^{-m(\sched,\tsched)}$.
This metric indeed turns the set of deterministic weight-based schedulers into a compact space as shown in \cite{fossacs2019}:

\begin{restatable}[Compactness of the space of deterministic weight-based schedulers \cite{fossacs2019}]{lemma}{lemmacompactness}
	\label{lem:compact}
	Let $\cM$ be as above. The space of all deterministic weight-based schedulers with the topology induced by the metric $d_{\cM}$ is compact.
\end{restatable}

\begin{theorem}[Existence of an optimal deterministic weight-based scheduler]
	Let $\cM$ and $\lambda>0$ be as above. There is a deterministic weight-based scheduler $\sched$ with
	\[
		\VPE[\lambda]_{\cM}^{\sched}= \VPE[\lambda]_{\cM}^{\max}.
	\]
\end{theorem}

\begin{proof}
	The claim follows from  Lemma \ref{lem:continuous}, Theorem \ref{thm:determinist_scheduler}, and Lemma \ref{lem:compact}, as continuous functions on compact spaces obtain their maximum.
\end{proof}

\subsection{Hardness of the threshold problem} \label{sub:hardness}

The result that the maximal variance-penalized expectation can be achieved by a deterministic scheduler can be used for the following hardness result:

\begin{restatable}{theorem}{theoremhardnessVPE}
	Given an MDP $\cM$ with non-negative weights and two rationals $\lambda,\vartheta>0$, deciding whether
	$\VPE[\lambda]^{\max}_{\cM} \geq \vartheta$ is EXPTIME-hard. Furthermore, for acyclic MDPs $\cM$, the problem is PSPACE-hard.
\end{restatable}

\begin{proof}[Proof sketch]
We reduce from the following problem which is shown to be EXPTIME-hard in general and PSPACE-hard for acyclic MDPs in \cite{HaaseKiefer15}:
	Given an MDP $\cM$ and a natural number $T>0$ such that $\goal$ is reached in~$\cM$ almost surely under all schedulers, decide whether
	there is a scheduler $\sched$ such that $\Pr^{\sched}_{\cM}(\rawdiaplus \goal{=T})=1$. 
	
	The idea is to construct an MDP $\cM^\prime$ that reaches $\goal$ with weight $T$ with probability $1/2$ directly and otherwise behaves like $\cM$.
	By choosing $\lambda$ sufficiently large, we can show that $\VPE[\lambda]^{\max}_{\cM^\prime} \geq T$ is only possible if and only if there is a scheduler with 
	$\Var^{\sched}_{\cM}(\rawdiaplus \goal)=0$. This scheduler then has to reach $\goal$ with weight $T$ on all paths. The necessary technical calculations can be found in Appendix \ref{app:hardness}.
\end{proof}

\subsection{Saturation Point} \label{sub:saturation}

\begin{figure}[t]
\centering
    \resizebox{.35\textwidth}{!}{%
      \begin{tikzpicture}[scale=1,auto,node distance=8mm,>=latex]
        \tikzstyle{round}=[thick,draw=black,circle]

        \node[round, draw=black,minimum size=10mm] (goal2) {$c$};
        \node[below=18mm of goal2] (input) {};
        \node[round, right=11mm of input, minimum size=10mm] (c) {$s$};
        \node[round, left=11mm of input, minimum size=10mm] (b) {$\sinit$};

        \node[round,above=11mm of goal2,minimum size=10mm] (goal) {$\goal$};

        \draw[color=black ,->,very thick] (goal2) edge [bend left]  node [pos=0.5,left] {$\alpha:+1$} (goal) ;
        \draw[color=black ,->,very thick] (goal2) edge [bend right]  node [pos=0.5,right] {$\beta:+0$} (goal) ;

        \draw[color=black ,->,very thick] (b)  edge  node [very near start, anchor=center] (m6) {} node [pos=0.2,left] {$1/2$} (goal2) ;
        \draw[color=black ,->,very thick] (b) edge  node [very near start, anchor=center] (m5) {} node [pos=0.22,below] {$1/2$} (c) ;
        \draw[color=black , very thick] (m5.center) edge [bend right=35] node [pos=0.5,right=1pt] {$+0$} (m6.center);

        \draw[color=black ,->,very thick] (c)  edge  node [very near start, anchor=center] (h5) {} node [pos=0.2,left] {$1/2$} (goal2) ;
        \draw[color=black ,->,very thick] (c) edge [loop, out=60, in=0, min distance=15mm]  node [pos=.07, anchor=center] (h6) {} node [pos=0.27,below=2pt] {$1/2$} (c) ;
        \draw[color=black , very thick] (h5.center) edge [bend left=35] node [pos=0.5,above=1pt] {$+1$} (h6.center);

      \end{tikzpicture}
    }

  \caption{The MDP $\cM$ used in Example \ref{exam:saturation}.}
  \label{fig:ex_saturation}
\end{figure}
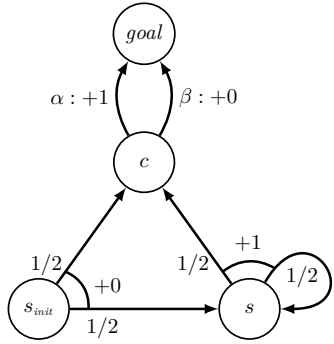
In the sequel, we will provide a series of results that allow us to further restrict the class of deterministic schedulers that we have to consider when  maximizing the variance-penalized expectation. In the end, we obtain a finite set of deterministic finite-memory schedulers among which there is a scheduler achieving the optimal variance-penalized expectation. In particular, this means that the optimum is computable.

The key step is the insight that we can provide a natural number $K$ computable in polynomial time such that an optimal scheduler $\sched$ for the variance-penalized expectation has to minimize the expected accumulated weight before reaching $\goal$ once a weight of at least $K$ has already been accumulated on a run. Furthermore, the behavior of $\sched$ after a weight of at least $K$ has been accumulated must minimize the variance of the weight that will still be accumulated
among all expectation-minimal schedulers. %These statements are made precise in the sequel.
We call this value $K$ a \emph{saturation point}.

\begin{example}\label{exam:saturation}

The MDP $\cM$ in Figure \ref{fig:ex_saturation} aims to provide some intuition on the results of this section.
The state $c$ in this MDP is reached with accumulated weight $n$ with probability $(1/2)^{n+1}$ for all $n\in \mathbb{N}$. Then, the choice has to be made whether to collect weight $+1$ or $0$ before moving to $\goal$.
We want to take a closer look at a family of special weight-based deterministic finite-memory schedulers for $\cM$: Let $\sched_k$ be the scheduler that chooses action $\alpha $ in $c$ if the accumulated weight is less than $k$ and otherwise chooses action $\beta$.

For these schedulers, we can explicitly provide expectation and variance:
The probability that the scheduler $\sched_k$ chooses $\alpha$ in $c$ is  $1-(1/2)^k$. As the expected accumulated weight before reaching $c$ is $1$,
we obtain a total expected accumulated weight of 
\[
\mathbb{E}^{\sched_k}_{\cM}(\rawdiaplus \goal)= 2- \frac{1}{2^k}.
\]
To obtain the variance, we compute $\mathbb{E}^{\sched_k}_{\cM}(\rawdiaplus \goal^2)$:
\begin{align*}
&\mathbb{E}^{\sched_k}_{\cM}(\rawdiaplus \goal^2)
=\sum_{n=0}^{k-1} \frac{1}{2^{k+1}} \cdot (n+1)^2 + \sum_{n=k}^{\infty} \frac{1}{2^{k+1}} \cdot n^2  \\
=&\sum_{n=0}^{\infty} \frac{1}{2^{k+1}} \cdot (n+1)^2 -  \sum_{n=k}^{\infty} \frac{1}{2^{k+1}} \cdot ( 2n+1 ) = 6-\frac{2k+3}{2^k}.
\end{align*}
We can then easily compute the variance
\[
\Var^{\sched_k}_{\cM}(\rawdiaplus \goal) = \mathbb{E}^{\sched_k}_{\cM}(\rawdiaplus \goal^2) -(\mathbb{E}^{\sched_k}_{\cM}(\rawdiaplus \goal))^2 
= 2+\frac{1-2k}{2^k}-\frac{1}{4^k}.
\]
 For $\lambda=1$, we obtain the following VPE:
\[
\VPE[\lambda]^{\sched_k}_{\cM} = \frac{k-1}{2^{k-1}} + \frac{1}{4^k}.
\]
Comparing scheduler $\sched_k$ to $\sched_{k+1}$, we obtain:
$
\VPE[\lambda]^{\sched_{k+1}}_{\cM}-\VPE[\lambda]^{\sched_{k}}_{\cM} = {2-k}/{2^k}-{3}/{4^k}$.
This difference is  negative for $k\geq 2$ and positive for  $k=1$. We conclude that among the schedulers $\sched_k$, the scheduler $\sched_2$ is VPE-optimal. Interestingly, this means  choosing not to accumulate the additional weight $+1$ by choosing $\alpha$ is better already for small amounts of accumulated weight.
Intuitively, the reason is that  choosing $\alpha$ for an accumulated weight $\geq 2$ has a larger effect on the variance than on the expectation.
Increasing the expectation in particular also increases the squared deviation of the path that reach $\goal$ with weight $1$ which has probability $1/2$ under $\sched_k$ for $k\geq 2$. The saturation point result of this section will tell us that an optimal scheduler always has to minimize the weight that is expected to still be accumulated weight once sufficiently much weight has already been accumulated.
\Ende
\end{example}

Let $\cM=(S,\Act,P,\sinit,\wgt,\goal)$ be an MDP without end components and with non-negative weights as above and let $\lambda>0$ be a rational.
Before we define $K$ and show that it can be computed in polynomial time, we need some additional notation.

For each state $s\in S$, define
$
	e_s\eqdef \mathbb{E}^{\min}_{\cM,s}(\rawdiaplus \goal)$.
For each state $s\in S\setminus \{\goal\}$, we define the subset $\Act^{\min}(s) \subseteq \Act(s)$ of actions allowing to minimize the expectation analogously to $\Act^{\max}$ before:
$
	\Act^{\min}(s)\eqdef \{\alpha\in \Act(s) \mid e_s = \wgt(s,\alpha) + \sum_{t\in S} P(s,\alpha, t)\cdot e_t \}$.
Choosing an action not belonging to $\Act^{\min}(s)$ in state $s$ ensures that the expected  accumulated weight before reaching $\goal$ is higher than the minimal possible value. Further, we can define the minimal amount by which choosing a non-minimizing action increases the expectation:
\[
	\delta\eqdef \min \{(\wgt(s,\alpha)+ \sum_{t\in S} P(s,\alpha, t)\cdot e_t ) - e_s
	\mid s\in S\setminus\{\goal\} \text{ and }\alpha\in \Act(s)\setminus\Act^{\min}(s)\}.
\]
If the set on the right hand side is empty, all schedulers minimize the expected accumulated weight before reaching $\goal$ and the claims of this section hold trivially. So, we can assume that this set is non-empty. By the definition of $\Act^{\min}$, we observe that $\delta>0$.

 Next, we can compute an upper bound $U_1$ for $\mathbb{E}^{\max}_{\cM,s}(\rawdiaplus \goal)$ for all states $s$ by computing the 
	      maximal value $U_1 \eqdef \max_{s\in S}\mathbb{E}^{\max}_{\cM,s}(\rawdiaplus \goal)$.

Finally, let $\varepsilon$ be the minimal transition probability present in $\cM$.
As $\cM$ has no end components, the only trap state $\goal$ is reached within $n\eqdef |S|$ steps
under each scheduler with probability at least $\varepsilon^n$.
Let $W$ be the largest weight  in $\cM$. Within $n$ steps at most a weight of $n\cdot W$ is accumulated.
We use these observations for the following two values:
\begin{itemize}
	\item
	      First, we can provide a value $B_{1/2}$ such that the probability that a weight above $B_{1/2}$ is accumulated under any scheduler is at most $1/2$:
	      For this, let $b_{1/2}$ be such that
	      $
		      ((1-\varepsilon^n))^{b_{1/2}} \leq 1/2$.
	      This is the case if and only if $b_{1/2}$ is at least
	      \[
		      \frac{\log(1/2)}{ \log(1-\varepsilon^n)} = -\frac{1}{ \log(1-\varepsilon^n)}< \frac{1}{\varepsilon^n}.
	      \]
	      So,
	      we can choose $b_{1/2}$
	      to be  $ \frac{1}{ \varepsilon^n}$.
	      Then, with probability at most $1/2$, a path has length at least $n\cdot b_{1/2}$. This allows us to defined $B_{1/2}\eqdef b_{1/2}\cdot n \cdot W$.
	\item
	      Second, we compute an upper bound $U_2$ for $\max_{s\in S}\mathbb{E}^{\max}_{\cM,s}(\rawdiaplus \goal^2)$: 
	      With probability $\varepsilon^n$ a path has weight at most $n\cdot W$; with probability $(1-\varepsilon^n) \cdot \varepsilon^n$ it has weight at most  $2\cdot n\cdot W$;
	      with probability $(1-\varepsilon^n)^2 \cdot \varepsilon^n$ it has weight at most  $3\cdot n\cdot W$; and so on. So, we get that
	      $
		      \max_{s\in S}\mathbb{E}^{\max}_{\cM,s}(\rawdiaplus \goal^2) \leq  \sum_{i=0}^{\infty} (1-\varepsilon^n)^{i}\cdot \varepsilon^n \cdot ((i+1)\cdot n\cdot W)^2$.
	      This allows us to define
	      \[
		      U_2 \eqdef \frac{2\cdot n^2\cdot W^2}{\varepsilon^{2n}} \geq \frac{(2-\varepsilon^n)\cdot n^2\cdot W^2}{\varepsilon^{2n}}=\sum_{i=0}^{\infty} (1-\varepsilon^n)^{i}\cdot \varepsilon^n \cdot ((i+1)\cdot n\cdot W)^2.
	      \]
\end{itemize}
We are now in the position to define the saturation point $K$:
Let $K$ be the least natural number with
\[
	K\geq B_{1/2} = \frac{n\cdot W}{\varepsilon^n} \qquad \text{and} \qquad K\geq \frac{U_1/\lambda +  U_2 +  2U_1 +U_1^2/2}{\delta}+1.
\]
The definition of $K$ is arguably a bit cumbersome, but the choices will become clear in the proof of Theorem \ref{thm:saturation}.
All values involved except for $\delta$ and $U_1$ can be computed directly from $n$, $W$, and $\varepsilon$ in polynomial time. The values $\delta$ and $U_1$ require to maximize or minimize the expected value of $\rawdiaplus \goal$ from all states, i.e., to solve an SSPP which can be done in polynomial time by linear programming \cite{BerTsi91,deAlf99}.

\begin{restatable}[Saturation point]{theorem}{theoremsaturation} \label{thm:saturation}
	Let $\cM$, $\lambda>0$ and $K$ be as above. 
	Let $\sched$ be a scheduler with  $\VP[\lambda]_{\cM}^{\sched}=\VP[\lambda]_{\cM}^{\max}$.
	Then,  for each finite $\sched$-path $\pi$ with $\wgt(\pi)\geq K$, the residual scheduler $\after{\sched}{\pi}$ satisfies
	\[
		\mathbb{E}^{\after{\sched}{\pi}}_{\cM,\last(\pi)} (\rawdiaplus \goal) = \mathbb{E}^{\min}_{\cM,\last(\pi)} (\rawdiaplus \goal).
	\]
\end{restatable}

\begin{proof}[Proof sketch]
Let $\sched$ be a scheduler with $\VP[\lambda]^{\sched}_{\cM} = \VP[\lambda]^{\max}_{\cM} $.
	Suppose there is a $\sched$-path $\pi^\prime$ with $\wgt(\pi^\prime)\geq K$ such that
	$
		\mathbb{E}^{\after{\sched}{\pi^\prime}}_{\cM,\last(\pi^\prime)} 
		(\rawdiaplus \goal) > \mathbb{E}^{\min}_{\cM,\last(\pi^\prime)} (\rawdiaplus \goal)	$.
	Then, there must be an $\sched$-path $\pi$ that extends $\pi^{\prime}$ such that $\sched$ chooses an action $\alpha\not\in \Act^{\min}(\last(\pi))$ with positive probability. 
	
	The residual scheduler $\tsched$ of $\sched$ after $\pi$ in case $\sched$ chooses $\alpha$ then satisfies $\mathbb{E}^{\tsched}_{\cM,\last(\pi)} (\rawdiaplus \goal )\geq \mathbb{E}^{\min}_{\cM,\last(\pi)} (\rawdiaplus \goal ) + \delta$. 
	We let $\sched^\prime$ be a scheduler that behaves like $\sched$ unless $\sched$ chooses $\alpha$ after $\pi$. In this case, $\sched^\prime$ minimizes the expected value of $\rawdiaplus \goal$ from then on.
	We  consider  the difference $\VPE[\lambda]^{\sched^{\prime}}_{\cM}- \VPE[\lambda]^{\sched}_{\cM}$. Using the bounds $U_1$ and $U_2$ and that the probability of $\pi$ is at most $1/2$ as $K\geq B_{1/2}$, we obtain a lower bound for this difference that consists of  an expression in terms of $U_1$, $U_2$, and $\lambda$ plus the term 
	\[
	\lambda\cdot \wgt(\pi)\cdot (\mathbb{E}^{\tsched}_{\cM,\last(\pi)} (\rawdiaplus \goal )- \mathbb{E}^{\min}_{\cM,\last(\pi)} (\rawdiaplus \goal )).
	\]
	Observing that this term is greater or equal to $\lambda \cdot K \cdot \delta$, the definition of $K$ was chosen exactly so that we can conclude that
	$\VPE[\lambda]^{\sched^{\prime}}_{\cM}- \VPE[\lambda]^{\sched}_{\cM}>0$. So, $\sched$ was not VPE-maximal yielding a contradiction.
\end{proof}

By the results of Section \ref{sec:minimal_variance}, there is a memoryless deterministic scheduler $\vsched$ that minimizes the variance among all schedulers minimizing the expected accumulated weight before reaching $\goal$.
More precisely,  for all states $s$, the scheduler $\vsched$ satisfies
	$
		\mathbb{E}^{\vsched}_{\cM,s}(\rawdiaplus\goal) = \mathbb{E}^{\min}_{\cM,s}(\rawdiaplus\goal)
	$
	and
	$
		\mathbb{V}^{\vsched}_{\cM,s} (\rawdiaplus \goal) = \inf_{\msched}\mathbb{V}^{\msched}_{\cM,s} (\rawdiaplus \goal)
	$
	where the infimum ranges over all schedulers $\msched$ with
	$\mathbb{E}^{\msched}_{\cM,s}(\rawdiaplus\goal) = \mathbb{E}^{\min}_{\cM,s}(\rawdiaplus\goal)$.
 We use the existence of this scheduler in the following theorem.

\begin{restatable}{theorem}{theoremoptimalaboveK}\label{thm:optimal_scheduler_above_K}
	Let $\cM$, $\lambda>0$, $K$, and $\vsched$ be as above. 
	Let $\sched$ be a deterministic scheduler with $\VP[\lambda]_{\cM}^{\sched}=\VP[\lambda]_{\cM}^{\max}$.
	Let $\tsched$  be the scheduler that agrees with $\sched$ on all paths $\pi$ with weight less than $K$ and that chooses actions according to the memoryless deterministic scheduler $\vsched$ after paths $\pi^\prime$ with $\wgt(\pi^\prime)\geq K$. This scheduler $\tsched$ satisfies
	$
		\VP[\lambda]_{\cM}^{\tsched}=\VP[\lambda]_{\cM}^{\max}$, too.
\end{restatable}

\begin{proof}[Proof sketch]
Given a scheduler $\sched$ with $\VP[\lambda]_{\cM}^{\sched}=\VP[\lambda]_{\cM}^{\max}$, and a path $\pi$ with $\wgt(\pi)\geq K$, we compare the scheduler $\sched$ to the scheduler $\sched^\prime$ that behaves like $\sched$, but switches to the behavior of $\vsched$ after $\pi$. We obtain that 
 $\VP[\lambda]_{\cM}^{\sched^\prime}\geq \VP[\lambda]_{\cM}^{\sched}$ is equivalent to $\mathbb{E}^{\vsched}_{\cM}(\rawdiaplus \goal^2)\leq \mathbb{E}^{\after{\sched}{\pi}}_{\cM}(\rawdiaplus \goal^2)$. This holds because
 $\mathbb{V}^{\vsched}_{\cM}(\rawdiaplus \goal)\leq \mathbb{V}^{\after{\sched}{\pi}}_{\cM}(\rawdiaplus \goal)$ as $\vsched$ and $\after{\sched}{\pi}$ achieve the same expectation. Using a continuity argument as before, we show that  changing the behavior of $\sched$ to $\vsched$ after all paths with weight at least $K$ does not decrease the VPE.
\end{proof}

Put together, we have shown that the maximal VPE is obtained by a weight-based deterministic scheduler that switches to the memoryless behavior of $\vsched$ as soon as a weight of at least $K$ has been accumulated, which also means that it uses only finite memory.

\subsection{Computation of the optimal VPE} \label{sub:computation}

%\tudparagraph{1ex}{Computing optimal variance-penalized expectations.}
%
%
%
Given an MDP $\cM=(S,\Act,P,\sinit,\wgt, \goal)$ with non-negative weights and without end components and $\lambda>0$ as before, 
let $K$ be the saturation point given above. Note that $K$ is computable in polynomial time and that hence its numerical value is at most exponential in the size of $\cM$.
We construct the following MDP $\cM^\prime$ that encodes the weights that are accumulated until the saturation point is exceeded into the state space:
Let $W$ be the maximal weight occurring in $\cM$. The set of states is $S^\prime = S\times \{0,1,\dots, K+W-1\}$. The set of actions remains unchanged.
The new probability transition function is given by 
$
P^\prime ((s,w),(t,w+\wgt(s,\alpha))) \eqdef P(s,\alpha,t)
$
for all $s,t\in S$, all $w<K$, and all $\alpha\in \Act(s)$. All remaining transition probabilities are $0$. Note that this means that all states of the form $(\goal,w)$ with $w\in \{0,1,\dots, K+W-1\}$ and of the form $(s,w)$ with $s\in S$ and $w\in \{K,K+1,\dots, K+W-1\}$ are trap states in $\cM^\prime$.
The initial state is $\sinit^\prime\eqdef (\sinit,0)$. The weight function is not relevant in $\cM^\prime$.

 Let $\vsched$ be the memoryless deterministic scheduler for $\cM$ as in Theorem \ref{thm:optimal_scheduler_above_K} that specifies the optimal behavior in order to maximize the variance-penalized expectation as soon as a weight of at least $K$ has been accumulated. Let us call the set of weight-based deterministic schedulers for $\cM$ that behave like $\vsched$ after a weight of at least $K$ has been accumulated by $\mathrm{WD}_K (\cM)$.
Clearly, there is a natural one-to-one-correspondence between memoryless deterministic schedulers for $\cM^\prime$ and schedulers in $\mathrm{WD}_K (\cM)$.

By the results of Section \ref{sec:minimal_variance}, for each state $s\in S$, we can compute the values
\[
e_s \eqdef \mathbb{E}^{\vsched}_{\cM,s}(\rawdiaplus \goal) \qquad \text{ and }\qquad q_s \eqdef \mathbb{E}_{\cM,s}^{\vsched}(\rawdiaplus \goal^2) = \Var^{\vsched}_{\cM,s}(\rawdiaplus \goal) + e_s^2
\]
in polynomial time. The following lemma now allows us to express the VPE in $\cM$ in terms of reachability probabilities in $\cM^\prime$.

\begin{restatable}{lemma}{lemmaVPEviaReachability}\label{lem:VPE_via_reachability_probabilities}
Let $\cM$, $\cM^\prime$, $K$ and $\lambda$ be as above. Given a scheduler $\sched\in \mathrm{WD}_{K}(\cM)$  also viewed as a memoryless deterministic scheduler for $\cM^\prime$, let
\[
\mu \eqdef \sum_{w<K} \Pr^{\sched}_{\cM^\prime} (\Diamond (\goal,w)) \cdot w + \sum_{s\in S, w\geq K} \Pr^{\sched}_{\cM^\prime} (\Diamond (s,w)) \cdot (w+e_s).  \tag{$\dagger$}
\]
Then,
\begin{align*}
\VPE[\lambda]^{\sched}_{\cM} =  \mu - \lambda \cdot \Big( & \sum_{w<K} \Pr^{\sched}_{\cM^\prime} (\Diamond (\goal,w)) \cdot (w-\mu)^2) \\
&+ \sum_{s\in S, w\geq K} \Pr^{\sched}_{\cM^\prime} (\Diamond (s,w)) \cdot  ((w-\mu)^2  + 2(w-\mu)e_s + q_s ) )
   \Big). \tag{$\ddagger$}
\end{align*}
\end{restatable}

\begin{proof}
It is clear that $\mu=\mathbb{E}^{\sched}_{\cM}(\rawdiaplus \goal)$. So,
we have to show that
\begin{align*}
&\Var^{\sched}_{\cM}(\rawdiaplus\goal ) = \mathbb{E}^{\sched}_{\cM} ((\rawdiaplus \goal - \mu)^2)\\
=& 
 \sum_{w<K} \Pr^{\sched}_{\cM^\prime} (\Diamond (\goal,w)) \cdot (w-\mu)^2) \\
&+ \sum_{s\in S, w\geq K} \Pr^{\sched}_{\cM^\prime} (\Diamond (s,w)) \cdot  ((w-\mu)^2  + 2(w-\mu)e_s + q_s ) ). \tag{$\circ$}
\end{align*}

The event $\Diamond (s,w)$ that $(s,w)$ is reached in $\cM^\prime$ corresponds to the event that a path in $\cM$ has a prefix of weight $w$ ending in $s$. We denote this event in $\cM$ also by $\Diamond (s,w)$.
If $\Pr^{\sched}_{\cM^\prime} (\Diamond (s,w))>0$ for $w\geq K$, then 
\begin{align*}
&\mathbb{E}^\sched_{\cM}((\rawdiaplus \goal -\mu)^2 | \Diamond (s,w) ) \\
=& \mathbb{E}^\vsched_{\cM,s}((\rawdiaplus \goal + w -\mu)^2 ) \\
=& (w-\mu)^2 + 2 (w-\mu)\cdot \mathbb{E}^{\vsched}_{\cM,s} (\rawdiaplus \goal) + \mathbb{E}^\vsched_{\cM,s}(\rawdiaplus \goal^2). 
\end{align*}
So, the sums in equation ($\circ$) sum up the conditional expectation of $(\rawdiaplus \goal- \mu)^2$ in $\cM$ under the conditions that  $(\goal,w)$ is reached for $w<K$ or that the state $s$ is the first one reached when the accumulated weight exceeds $K$ with weight $w\geq K$, multiplied by the respective probabilities of the conditions.
\end{proof}

Putting everything together, we arrive at the main result.

\begin{theorem}
Let $\cM$ and $\lambda$ be as above. 
Given a rational $\vartheta$, the threshold problem whether $\VPE[\lambda]^{\max}_{\cM}\geq \vartheta$ is in NEXPTIME.
The optimal value $\VPE[\lambda]^{\max}_{\cM}$ and an optimal scheduler can be computed in exponential space.
\end{theorem}

\begin{proof}
The threshold problem can be decided in non-deterministic exponential time as follows:
Given $\cM$ and $\lambda$, compute $K$ in polynomial time and construct $\cM^\prime$ as above (of exponential size) in exponential time. Guess a memoryless deterministic scheduler $\sched$ for $\cM^\prime$ also viewed as a scheduler in $\mathrm{WD}_K(\cM)$. The reachability probabilities for all trap states in $\cM^\prime$ under $\sched$ can then be computed in time polynomial in the size of $\cM^\prime$. With the help of equations ($\dagger$) and ($\ddagger$) from Lemma \ref{lem:VPE_via_reachability_probabilities}, $\VPE[\lambda]^\sched_{\cM}$ can be computed from these reachability probabilities in time polynomial in the size of $\cM^\prime$. If $\VPE[\lambda]^\sched_{\cM}\geq \vartheta$, accept.
By Theorem \ref{thm:optimal_scheduler_above_K}, $\VPE[\lambda]^{\max}_{\cM}\geq \vartheta$ iff there is a scheduler  $\sched$ in $\mathrm{WD}_K(\cM)$ with $\VPE[\lambda]^\sched_{\cM}\geq \vartheta$. Due to the one-to-one correspondence between schedulers in $\mathrm{WD}_K(\cM)$
and memoryless deterministic schedulers for $\cM^\prime$, this establishes the correctness of the algorithm.

To compute the optimal value $\VPE[\lambda]^{\max}_{\cM}$, we compute $\VPE[\lambda]^\sched_{\cM}$ for all schedulers $\sched$ in $\mathrm{WD}_K(\cM)$ in the same fashion and always store the highest value found so far. As the memoryless schedulers for $\cM^\prime$ have an exponentially large representation, this can be done in exponential space and the optimal scheduler can be returned as well.
\end{proof}

\section{Conclusion}
In our results, there remains a complexity gap between the EXPTIME-lower bounds and the  exponential-space and NEXPTIME-upper bounds for the optimization of the VPE in MDPs with non-negative weights and the corresponding threshold problem, respectively.
Here, we want to shed some light on this complexity gap:
It is well-known that the possible vectors of expected frequencies of all states in an MDP can be characterized by a linear equation system (see, e.g., \cite{Kallenberg}). 
Using this linear equation system for the exponentially large MDP constructed in Section \ref{sub:computation}
and equations ($\dagger$) and ($\ddagger$) from that section, the threshold problem for the maximal VPE can be reformulated as the satisfiability  
problem of an exponentially sized system of quadratic inequalities. The optimization problem can likewise be formulated as an exponentially large
\emph{quadratically constrained quadratic program} (QCQP).
This satisfiability problem and QCQPs are NP-hard in general. 
The question whether the inequality system of exponential size we obtain here has a special structure which allows it to be solved in exponential time remains open here.

This observation stands in contrast to 
conceptually similar saturation point results straight-forwardly leading to
 exponential time algorithms   (see, e.g.,  \cite{icalp2020}).
For example, the threshold problem for conditional expectations: ``Given a set $T\subseteq S$, is there a scheduler $\sched$ with
$\mathbb{E}^{\sched}_{\cM}(\rawdiaplus \goal \mid \Diamond T)\geq \vartheta$?'' admits a saturation point result in MDPs with non-negative weights as well 
 \cite{tacas2017}.
Deriving a 
 system of inequalities as above, however, leads to a system of linear inequalities after straight-forward transformations.
Hence, this approach directly leads to an exponential time algorithm for the threshold problem 
for conditional expectations.
For the VPE, the system of inequalities seems to be inherently of a polynomial nature which can be seen as an indication that the situation here  is fundamentally more difficult.

Further, we restricted our attention to MDPs with non-negative weights.
When allowing positive and negative weights, the key result, the existence of a saturation point, does not hold anymore.
For conditional expectations and other problems relying on the existence of a saturation point, the switch to  integer weights makes the problems even at least as hard as the Positivity problem for linear recurrence sequences, a number theoretic problem whose decidability has been open for many decades (see \cite{icalp2020,piribauer2021thesis}).
The question whether such a hardness result for the threshold problem of the VPE, rendering decidability impossible without a  breakthrough in  number theory, can be established remains as future work.

Further possible directions of research include the investigation of the following 
multi-objective threshold problem: Given $\eta$ and $\nu$, is there a scheduler with expectation at least $\eta$ and variance at most $\nu$?
As the variance treats good and bad outcomes symmetrically, replacing the variance in the VPE by a one-sided deviation measure, 
such as the lower semi-variance that only takes the outcomes worse than the expected value into account, constitutes
another natural  extension of this work.

%%%%%%%%%%%%%%%%%%%%%%%%%%%%%%%%%%%%%%%%%%%%%%%%%%%%%%%%%%%%%%%%%%%%%%
%%%%% 
%%%%%     references 
%%%%%
%%%%%%%%%%%%%%%%%%%%%%%%%%%%%%%%%%%%%%%%%%%%%%%%%%%%%%%%%%%%%%%%%%%%%%

\clearpage

\bibliographystyle{plain}
\bibliography{references/lit}

%%%%%%%%%%%%%%%%%%%%%%%%%%%%%%%%%%%%%%%%%%%%%%%%%%%%%%%%%%%%%%%%%%%%%%%

\clearpage
\begin{appendix}

\section{Omitted proofs of Section \ref{sec:minimal_variance}}\label{app:minimal_variance}

\lemmapreprocessing*

\begin{proof}[Proof sketch]
    % Consider a $0$-EC~$\mathcal{E}$.
    %Recall that a property~$\phi$ is~$\mathcal{E}$-invariant
    %if 1) paths that end with an infinite suffix inside~$\mathcal{E}$ do not belong to~$\phi$,
    %and 2) and if two paths $\xi,\xi'$ have the same projection to~$\cM'$, then
    %$\xi \models \phi$ iff $\xi' \models \phi$.
    The so-called spider construction of~\cite[Lemma 3.7]{BaiBerDubGbuSan-LICS18}
    was shown to produce an MDP~$\cM'$ without $0$-ECs, and which provides the mappings $f$ and $g$,
    preserving, in particular, the probabilities of all properties $\phi_k$ defined by the set of paths that reach~$\goal$
    with cost at least~$k$.
    The construction actually preserves all~\emph{$\mathcal{E}$-invariant} properties
    where~$\mathcal{E}$ is the eliminated $0$-EC; and it was shown that~$\phi_k$ are such properties
    \footnote{Note that in item S3.2 of \cite[Lemma 3.7]{BaiBerDubGbuSan-LICS18},
        $p_s^\sched = 0$ if~$\sched$ is proper.}.
    The equalities of all moments, and in particular, the expectation and variance immediately follow.
\end{proof}

\lemmapruned*

\begin{proof}
    Towards a contradiction, assume that~$\cM'$ has an end-component~$\mathcal{E}$.
    Let~$s$ be a state in~$\mathcal{E}$, $\sched$ a memoryless deterministic scheduler that stays forever
    in~$\mathcal{E}$.
    Consider~$k>0$, let $B_k$ denote the set of $\sched$-paths of~$\mathcal{E}$ of length~$k$ that start
    at~$s$. Using that all actions in $\cE$ belong to $\Act^{\max}$, we can write
    \begin{align*}
        \mu_s & = \sum_{\rho \in B_k} \Pr^{\sched}_{\cE}(\rho)(\wgt(\rho) + \mu_{\last(\rho)})                         \\
              & = \sum_{\rho \in B_k} \Pr^{\sched}_{\cE}(\rho)\wgt(\rho) + \sum_{\rho \in B_k}\Pr^{\sched}_{\cE}(\rho)\mu_{\last(\rho)} \\
              & \leq \sum_{\rho \in B_k} \Pr^{\sched}_{\cE}(\rho)\wgt(\rho) + \max_{t \in S} \mu_{t}.                  \\
    \end{align*}
    More rigorously, these equations follow easily by induction on $k$ using the definition of $\Act^{\max}$.
    Notice that $\sum_{\rho \in B_k} P(\rho)\wgt(\rho)$ is the expected accumulated weight
    over~$k$ steps. However, we know that $\mathbb{E}^{\max}_{\mathcal{E}}(\MeanPayoff)<0$,
    so the expected mean payoff inside this end component is also negative.
    This implies that $\lim_{k \rightarrow \infty}\sum_{\rho \in B_k} \Pr^{\sched}_{\cE}(\rho)\wgt(\rho) = -\infty$
    which contradicts the above inequality. Thus, $\cM'$ has no end-components.
      It also follows that all schedulers are proper.

    We now show that all schedulers are expectation-optimal by 
    first observing that all memoryless deterministic schedulers are expectation-optimal.
    In fact, if~$\sched_{\MD}$ denotes such a scheduler in~$\cM'$,
    then $(\mathbb{E}^{\sched_{\MD}}_{\cM',s}(\rawdiaplus\goal))_{s \in S}$ is a solution of~\eqref{eqn:mu} and hence agrees with $(\mathbb{E}^{\max}_{\cM',s}(\rawdiaplus\goal))_{s \in S}$.
    As maximal and minimal expectation of $\rawdiaplus\goal$ are obtained by memoryless deterministic schedulers \cite{BerTsi91}, all schedulers have the same expected values from each state.
\end{proof}

\lemmavarianceequation*

\begin{proof}
    We consider the random variables $X_0,X_1,\ldots$ which are the $i$-th visited state;
    $R$ which is the sum of all weights until~$\goal$ is reached (i.e. this stands for~$\rawdiaplus\goal$),
    and~$R'$ the sum of all weights but the first one until~$\goal$ is reached.
    \begin{align*}
        \mathbb{V}_{\cM,s}^{\inf} & = \inf_{\sched} \mathbb{E}_{\cM,s}^{\sched}[(R-\mu_s)^2]                                                                                          \\
                                  & =\min_{a \in \Act(s)} \inf_{\sched : \sched(s) = a}\mathbb{E}_{\cM,s}^{\sched}[(R-\mu_s)^2]                                                       \\
                                  & =\min_{a \in \Act(s)} \inf_{\sched : \sched(s) = a}\mathbb{E}_{\cM,s}^{\sched}[(\wgt(s,a) + R' -\mu_s)^2]                                         \\
                                  & =\min_{a \in \Act(s)} \inf_{\sched : \sched(s) = a}\sum_{t \in S}P(s,a,t) \mathbb{E}_{\cM,s}^{\sched}[(\wgt(s,a) + R' -\mu_s)^2 \mid X_1=t]       \\%&& \parbox[t]{4cm}{(by $\mathbb{E}[(c+X)^2]) = (c+\mathbb{E}[X])^2 + \mathbb{V}[X]$)} \\
                                  & =\min_{a \in \Act(s)} \inf_{\sched : \sched(s) = a}\sum_{t \in S}P(s,a,t) \mathbb{E}_{\cM,t}^{\sched}[(\wgt(s,a) + R -\mu_s)^2]                   \\%&& \parbox[t]{4cm}{(by $\mathbb{E}[(c+X)^2]) = (c+\mathbb{E}[X])^2 + \mathbb{V}[X]$)} \\
        % &=\min_{a \in \Act(s)} \inf_{\sched : \sched(s) = a}\sum_{t \in S}P(s,a,t)\left(\wgt(s,a) -\mu_s + \mathbb{E}_{\cM,s}^{\sched}[R' \mid X_1=t]\right)^2 + \mathbb{V}_{\cM,s}^{\sched}[R'\mid X_1=t]\\
        %&=\min_{a \in \Act(s)} \inf_{\sched} \sum_{t \in S}P(s,a,t) \left(\wgt(s,a) -\mu_s + \mathbb{E}_{\cM,t}^{\sched}[R]\right)^2 + \mathbb{V}_{\cM,t}^{\sched}[R]\\
                                  & =\min_{a \in \Act(s)} \sum_{t \in S}P(s,a,t) \left(\left(\wgt(s,a) -\mu_s + \mu_t\right)^2 + \inf_{\sched} \mathbb{V}_{\cM,t}^{\sched}[R]\right). \\
        %\qquad (\text{by }\mathbb{E}[(c+X)^2]) = (c+\mathbb{E}[X])^2 + \mathbb{V}[X])\\
    \end{align*}

    Here, we used $\mathbb{E}[(c+X)^2]) = (c+\mathbb{E}[X])^2 + \mathbb{V}[X]$
    and the fact that $\mu_t = \mathbb{E}_{\cM,t}^{\sched}[R]$ is constant on the last line.
    This shows that $\mathbb{V}_{\cM,s}^{\inf}[R]$ satisfies the given equation.

    Notice that this has a form similar to \eqref{eqn:mu}, where the weight function
    not only depends on the state~$s$ and chosen action~$a$ but also on the next state~$t$,
    that is, $\wgt'(s,a,t) = (\wgt(s,a) - \mu_s +\mu_t)^2$. Alternatively,
    the dependence of the weight on the state~$t$ can also be modeled using intermediary states.
    Since $\cM$ has no end-components, this has a unique solution by~\cite{BerTsi91}.
\end{proof}

\section{Omitted proofs of Section \ref{sec:VPE}}\label{app:VPE}

\subsection{Existence of optimal deterministic schedulers}

\lemparabolic*

\begin{proof}
	As the expected values of $\rawdiaplus \goal$ and $\rawdiaplus \goal^2$ depend linearly on the expected frequencies of the state-weight pairs $(\goal,w)$ with $w\in\mathbb{N}$, we know that
	\[
		\mathbb{E}^{\rsched}_{\cM}(\rawdiaplus\goal) = p\cdot \mathbb{E}^{\sched}_{\cM}(\rawdiaplus\goal)+(1-p)\cdot \mathbb{E}^{\tsched}_{\cM}(\rawdiaplus\goal)
	\]
	and
	\[
		\mathbb{E}^{\rsched}_{\cM}(\rawdiaplus\goal^2) = p\cdot \mathbb{E}^{\sched}_{\cM}(\rawdiaplus\goal^2)+(1-p)\cdot \mathbb{E}^{\tsched}_{\cM}(\rawdiaplus\goal^2).
	\]
	The claim follows by basic arithmetic using
	$\Var^{\rsched}_{\cM}(\rawdiaplus\goal)=\mathbb{E}^{\rsched}_{\cM}(\rawdiaplus\goal^2)-\mathbb{E}^{\rsched}_{\cM}(\rawdiaplus\goal) ^2$:
	\begin{align*}
		       & \Var^{\rsched}_{\cM}(\rawdiaplus\goal)                                                                                                                                                                                             \\
		{}= {} & p\cdot \mathbb{E}^{\sched}_{\cM}(\rawdiaplus\goal^2)+(1-p)\cdot \mathbb{E}^{\tsched}_{\cM}(\rawdiaplus\goal^2) - (p\cdot \mathbb{E}^{\sched}_{\cM}(\rawdiaplus\goal)+(1-p)\cdot \mathbb{E}^{\tsched}_{\cM}(\rawdiaplus\goal)
		)^2                                                                                                                                                                                                                                         \\
		{}={}  & p\cdot \mathbb{E}^{\sched}_{\cM}(\rawdiaplus\goal^2)+(1-p)\cdot \mathbb{E}^{\tsched}_{\cM}(\rawdiaplus\goal^2)                                                                                                                     \\
		       & - \Big(p^2\mathbb{E}^{\sched}_{\cM}(\rawdiaplus\goal)^2 + 2p(1-p)\mathbb{E}^{\sched}_{\cM}(\rawdiaplus\goal)\cdot \mathbb{E}^{\tsched}_{\cM}(\rawdiaplus\goal) + (1-p)^2\mathbb{E}^{\tsched}_{\cM}(\rawdiaplus\goal)^2\Big)        \\
		       & \underbrace{+ (p^2-p)\mathbb{E}^{\sched}_{\cM}(\rawdiaplus\goal)^2 - (p^2-p)\mathbb{E}^{\sched}_{\cM}(\rawdiaplus\goal)^2}_{=0}                                                                                                    \\
		       & \underbrace{+ ((1-p)^2-(1-p))\mathbb{E}^{\tsched}_{\cM}(\rawdiaplus\goal)^2 - ((1-p)^2-(1-p))\mathbb{E}^{\tsched}_{\cM}(\rawdiaplus\goal)^2}_{=0}                                                                                  \\
		{}={}  & p\cdot \mathbb{E}^{\sched}_{\cM}(\rawdiaplus\goal^2) - p\cdot \mathbb{E}^{\sched}_{\cM}(\rawdiaplus\goal)^2 +(1-p)\cdot \mathbb{E}^{\tsched}_{\cM}(\rawdiaplus\goal^2) - (1-p)\cdot \mathbb{E}^{\tsched}_{\cM}(\rawdiaplus\goal)^2 \\
		       & - 2p(1-p)\mathbb{E}^{\sched}_{\cM}(\rawdiaplus\goal)\cdot \mathbb{E}^{\tsched}_{\cM}(\rawdiaplus\goal)                                                                                                                             \\
		       & - (p^2-p)\mathbb{E}^{\sched}_{\cM}(\rawdiaplus\goal)^2 - ((1-p)^2-(1-p))\mathbb{E}^{\tsched}_{\cM}(\rawdiaplus\goal)^2                                                                                                             \\
		{}={}  & p\cdot \Var^{\sched}_{\cM}(\rawdiaplus\goal)+(1-p)\cdot \Var^{\tsched}_{\cM}(\rawdiaplus\goal)                                                                                                                                     \\
		       & -2p(1-p)\mathbb{E}^{\sched}_{\cM}(\rawdiaplus\goal)\cdot \mathbb{E}^{\tsched}_{\cM}(\rawdiaplus\goal)                                                                                                                              \\
		       & - p\cdot (p-1) \cdot \mathbb{E}^{\sched}_{\cM}(\rawdiaplus\goal)^2 - ((1-p)-1)\cdot(1-p) \cdot \mathbb{E}^{\tsched}_{\cM}(\rawdiaplus\goal)^2                                                                                      \\
		{}={}  & p\cdot \Var^{\sched}_{\cM}(\rawdiaplus\goal)+(1-p)\cdot \Var^{\tsched}_{\cM}(\rawdiaplus\goal)                                                                                                                                     \\
		       & + p\cdot (1-p) \cdot \Big(\mathbb{E}^{\sched}_{\cM}(\rawdiaplus\goal)^2 -2 \mathbb{E}^{\sched}_{\cM}(\rawdiaplus\goal)\cdot \mathbb{E}^{\tsched}_{\cM}(\rawdiaplus\goal)  + \mathbb{E}^{\tsched}_{\cM}(\rawdiaplus\goal)^2 \Big)   \\
		{}={}  & p\cdot \Var^{\sched}_{\cM}(\rawdiaplus\goal)+(1-p)\cdot \Var^{\tsched}_{\cM}(\rawdiaplus\goal) + p\cdot (1-p)\cdot (\mathbb{E}^{\sched}_{\cM}(\rawdiaplus\goal) - \mathbb{E}^{\tsched}_{\cM}(\rawdiaplus\goal))^2. \qedhere
	\end{align*}
\end{proof}

\lemmacontinuity*

\begin{proof}
As $\cM$ has no end components, the values $E_1\eqdef \max_{s\in S} \mathbb{E}^{\max}_{\cM,s}(\rawdiaplus \goal)$ and $E_2\eqdef \max_{s\in S} \mathbb{E}^{\max}_{\cM,s}(\rawdiaplus \goal^2)$ are finite (an explicit bound on the latter is also given  in Section \ref{sub:saturation}).
Furthermore, let $W$ be the maximal weight occurring in $\cM$, let $n$ be the number of states of $\cM$, and let $\delta$ be the minimal non-zero probability occurring in $\cM$. From any state, $\goal$ is reached within $n$ steps with probability at least $\delta^n$. Let $p\eqdef (1-\delta^n)$.
So, for any $k\in \mathbb{N}$, the probability that a weight of more than $n\cdot W \cdot k$ is accumulated under any scheduler is bounded by $p^n$.

We now show that $\mathbb{E}^\cdot_{\cM} (\rawdiaplus \goal)$ and $\mathbb{E}^\cdot_{\cM} (\rawdiaplus \goal^2)$ as functions from schedulers to real numbers are continuous in the sense of the lemma. The claim then follows immediately.

For the former, let $\varepsilon>0$ be given. Let $k_\varepsilon$ be such that $p^{k_\varepsilon} \cdot E_1 < \epsilon$. 
Now, let $\sched$ and $\tsched$ be two weight-based schedulers that agree on all state-weight pairs with weight at most $N_{\varepsilon}\eqdef k_{\varepsilon}\cdot n \cdot W$.
If a weight of more than $N_{\varepsilon}$ cannot be accumulated under $\sched$ and $\tsched$, there is nothing to show.
Otherwise,
\begin{align*}
 & | \mathbb{E}^{\sched}_{\cM}(\rawdiaplus \goal) - \mathbb{E}^{\tsched}_{\cM}(\rawdiaplus \goal) | \\
 {}\leq {}&
 \Pr^{\sched}_{\cM}(\rawdiaplus \goal >N_{\varepsilon})  \cdot | \mathbb{E}^{\sched}_{\cM}(\rawdiaplus \goal | \rawdiaplus \goal >N_{\varepsilon}) - 
 \mathbb{E}^{\tsched}_{\cM}(\rawdiaplus \goal | \rawdiaplus \goal >N_{\varepsilon}) | \\
{} \leq  {} & p^{k_{\varepsilon}} \cdot E_1 < \varepsilon.
\end{align*}

For the function $\mathbb{E}^\cdot_{\cM} (\rawdiaplus \goal^2)$, the claim follows analogously using the bound $E_2$.
\end{proof}

\lemmaRandomizationConvex*

In the proof, we use the following notation: Let  $\sched$ be a weight-based scheduler, $(s,w)$  a state-weight pair, and $\pi$ a path with $\last(\pi)=s$ and $\wgt(\pi)=w$. Then, $\after{\sched}{(s,w)}$ denotes the scheduler $\after{\sched}{\pi}$. As $\sched$ is weight-based, this definition does not depend on the choice of the path $\pi$.

\begin{proof}
    To obtain a candidate, for the value $p$, we first consider
    the expected frequency $\vartheta^{\sched}_{s,w}$, $\vartheta^{\sched_{\alpha}}_{s,w}$, and $\vartheta^{\sched_\beta}_{s,w}$,
    of the state-weight pair $(s,w)$ under $\sched$, $\sched_\alpha$, and $\sched_\beta$, respectively.
    Under all three schedulers, the probability $t>0$ that the state-weight pair $(s,w)$ is reached is the same.
    The expected frequencies now depend on the probability that a run returns to $(s,w)$ after leaving $(s,w)$.
    Let $r_{\alpha}$ be the probability that $\after{\sched_{\alpha}}{(s,w)}$ returns to state $s$ without accumulating additional weight, i.e., the probability that a run under $\after{\sched_{\alpha}}{(s,w)}$ starting in $s$ has a prefix $\pi$ of length $>1$ with $\last(\pi)=s$ and $\wgt(\pi)=0$.
    Let $r_{\beta}$ be defined analogously.
    The probability that $\sched$ returns from $(s,w)$ to this same state-weight pair is then $q\cdot r_{\alpha}+(1-q)\cdot r_{\beta}$.
    Note that $r_{\alpha}$ and $r_{\beta}$ are less than $1$ as all end components have negative maximal expected mean-payoff in $\cM$.
    We obtain the following expected frequencies of $(s,w)$:
    \begin{align*}
        \vartheta^{\sched_\alpha}_{s,w} & = t\frac{1}{1-r_{\alpha}}, \qquad  \vartheta^{\sched_\beta}_{s,w}= t\frac{1}{1-r_{\beta}}, \\
        \vartheta^{\sched}_{s,w}        & = t\frac{1}{1-q r_{\alpha}- (1-q) r_{\beta}}.
    \end{align*}
    The value $p$ has to satisfy $p\cdot \vartheta^{\sched_\alpha}_{s,w} + (1-p)\cdot \vartheta^{\sched_\beta}_{s,w}=\vartheta^{\sched}_{s,w}$ which is equivalent to
    \[
        \frac{1}{1-q r_{\alpha}- (1-q) r_{\beta}} = p\cdot \frac{1}{1-r_{\alpha}} +(1-p)\cdot \frac{1}{1-r_{\beta}}. \tag{$\diamond$}
    \]
    The value $\frac{1}{1-q r_{\alpha}- (1-q) r_{\beta}}$ lies between $ \frac{1}{1-r_{\alpha}}$ and $\frac{1}{1-r_{\beta}}$.
    If $r_\alpha=r_\beta$, any value $p\in [0,1]$ satisfies the equation. In this case, we also have $\vartheta_{s,w}^{\sched}=\vartheta_{s,w}^{\sched_\alpha}=\vartheta_{s,w}^{\sched_\beta}$.
    Otherwise, there is a unique value $p\in (0,1)$ satsifying ($\diamond$).

    Next, we have to check sure that also actions $\alpha$ and $\beta$ are chosen with the correct expected frequency in
    $(s,w)$ under $p\cdot \sched_\alpha \oplus (1-p)\cdot \sched_{\beta}$.
    Note that $\vartheta_{s,w,\alpha}^{\sched}=q\cdot \vartheta_{s,w}^{\sched}$ and that
    $\vartheta_{s,w,\alpha}^{\sched_\alpha}=\vartheta_{s,w}^{\sched_\alpha}$
    and $\vartheta_{s,w,\beta}^{\sched_\alpha}=0$, and analogously for $\sched_\beta$.

    If $r_{\alpha}=r_{\beta}$, we can simply choose $p=q$. If $r_{\alpha}\not = r_{\beta}$, we have to check that
    \[
        q\cdot  t\cdot \frac{1}{1-q r_{\alpha}- (1-q) r_{\beta}} =  p\cdot t \cdot \frac{1}{1-r_{\alpha}}.
    \]
   This equation follows from ($\diamond$) by basic arithmetic for $r_{\alpha}\not = r_{\beta}$ by multiplying out equation ($\diamond$) and dividing by $(r_{\alpha}-r_{\beta})$.

    So, indeed there exists a unique value $p$ such that the expected  frequencies of $(s,w)$ and of the expected number of times $\alpha$ and $\beta$, respectively, are chosen at $(s,w)$ agree under $\sched$ and $p\cdot \sched_\alpha \oplus (1-p)\cdot \sched_{\beta}$.
    As schedulers $\sched$, $\sched_{\alpha}$, and $\sched_{\beta}$ behave identically at all  state-weight pairs different to $(s,w)$,
    the expected frequencies of all other state weight pairs under $\sched$ and $p\cdot \sched_\alpha \oplus (1-p)\cdot \sched_{\beta}$ are  the same as well.
\end{proof}

\theoremsupremumdeterministic*

\begin{proof}
	Let $\sched$ be a scheduler that is not deterministic. By Corollary \ref{cor:weight-based}, we can assume w.l.o.g. that $\sched$ is weight-based. So, we view $\sched$ as a function from $S\times\mathbb{N}$ to probability distributions over actions.
	There is now a reachable state-weight pair $(s,w)$ at which $\sched$ schedules an action $\alpha$ with probability $0<q<1$. In a first step, we show that we can modify the scheduler such that it chooses $\alpha$ with probability $0$ or $1$ without decreasing the variance-penalized expectation.

	First, we suppose that the scheduler only chooses one other action $\beta$ at $(s,w)$ with positive probability to keep notation simpler. We explain how to treat the case with multiple actions afterwards.
	So, the  scheduler $\sched$ chooses $\alpha$ at $(s,w)$ with probability $q$ and  $\beta$ with probability $1-q$.
	Let $\sched_{\alpha}$ be the scheduler that agrees with $\sched$ on all state-weight pairs except for $(s,w)$ where it chooses $\alpha$ with probability $1$. Let $\sched_\beta$ be defined analogously.
	By Lemma \ref{lem:ranomization_convex}, there is a $p\in (0,1)$ such that $\sched$ and $p\cdot \sched_{\alpha} + (1-p)\cdot \sched_\beta$ lead to the same expected frequency of all state-weight pairs.
	Using Lemma \ref{lem:parabolic}, we obtain:
	\begin{align*}
		\VPE[\lambda]^{\sched}_{\cM}  = {} & \mathbb{E}^{\sched}_{\cM}(\rawdiaplus \goal) - \lambda \Var^{\sched}_{\cM}(\rawdiaplus \goal)                                           \\
		{} = {}                            & p\cdot ( \mathbb{E}^{\sched_\alpha}_{\cM}(\rawdiaplus \goal) - \lambda \Var^{\sched_\alpha}_{\cM}(\rawdiaplus \goal) )
		+ (1-p)\cdot ( \mathbb{E}^{\sched_\beta}_{\cM}(\rawdiaplus \goal) - \lambda \Var^{\sched_\beta}_{\cM}(\rawdiaplus \goal) )                                                   \\
		                                   & - \lambda \cdot p\cdot (1-p) (\mathbb{E}^{\sched_\alpha}_{\cM}(\rawdiaplus \goal)-\mathbb{E}^{\sched_\beta}_{\cM}(\rawdiaplus \goal))^2 \\
		{} \leq {}                         & p\cdot \VPE[\lambda]^{\sched_\alpha}_{\cM} + (1-p)\cdot \VPE[\lambda]^{\sched_\beta}_{\cM}.
	\end{align*}
	We conclude that $ \VPE[\lambda]^{\sched_\alpha}_{\cM} \geq  \VPE[\lambda]^{\sched}_{\cM} $ or $ \VPE[\lambda]^{\sched_\beta}_{\cM} \geq  \VPE[\lambda]^{\sched}_{\cM} $. So, indeed a scheduler that behaves like $\sched$ except at $(s,w)$ where it chooses $\alpha$ with probability $0$ or $1$ achieves at least the same variance-penalized expectation.

	We assumed for simplicity that the scheduler $\sched$ chooses only two actions with positive probability.
	In general,
	the scheduler $\sched$ might choose action $\alpha$ with probability $0<q<1$ and further  actions $\beta_1, \dots ,\beta_{\ell}\in \Act(\last(\pi))\setminus\{\alpha\}$ for some $\ell$ with positive probabilities $(1-q)\cdot \Delta(\beta_1),\dots,(1-q)\cdot \Delta(\beta_\ell)$ where $\Delta$ is a probability distribution over $ \Act(\last(\pi))\setminus\{\alpha\}$.
	The argument above still works if we  use $\beta$ as an abbreviation for choosing actions $\beta_1, \dots ,\beta_{\ell}\in \Act(\last(\pi))\setminus\{\alpha\}$ according to the probability distribution $\Delta$. Our proof then shows that at least one of the schedulers choosing $\alpha$ with probability $1$ or the remaining actions according to $\Delta$ does not decrease the variance-penalized expectation.
	In the latter case if $\ell>1$, the argument can successively be repeated by letting action $\beta_1$ take the role of $\alpha$. After at most $\ell$ changes to the scheduler, we find a scheduler that chooses one of the actions $\alpha$ and $\beta_1,\dots,\beta_{\ell}$ with probability $1$ at $(s,w)$ and otherwise behaves like $\sched$.

	To obtain a deterministic scheduler, we enumerate all (countably many) state-weight pairs $(s_0,w_0), (s_1,w_1), \dots$
	Let $\sched_0\eqdef \sched$. Once scheduler $\sched_i$ is defined, we let $\sched_{i+1}$ be a scheduler that makes a deterministic choice at the state-weight pair  $(s_i,w_i)$ and otherwise behaves like $\sched_i$ and that satisfies
	\[
		\VPE[\lambda]^{\sched_{i+1}}_{\cM}\geq \VPE[\lambda]^{\sched_i}.
	\]
	In the limit, we obtain a well-defined deterministic weight-based scheduler $\tsched$: The choice of $\tsched$ at any state-weight pair $(s_i,w_i)$ is given by
	$\sched_{i+1}(s_i,w_i)$ which is equal to $\sched_{j}(s_i,w_i)$ for all $j>i$.

	We claim that
	\[
		\VPE[\lambda]^{\tsched}_{\cM}\geq \VPE[\lambda]^{\sched}.
	\]
	To see this, let $\varepsilon>0$ be arbitrary and let $N_\varepsilon$ be the natural number as in Lemma \ref{lem:continuous}.
	There are only finitely many state-weight pairs $(s,w)$ with $w\leq N_\varepsilon$. Let $k_\varepsilon$ be a natural number such that all state-weight pairs $(s,w)$ with $w\leq N_\varepsilon$ occur before $(s_{k_{\varepsilon}},w_{k_{\varepsilon}})$ in our enumeration.
	Then $\sched_{k_{\varepsilon}}$ and $\tsched$ agree on all all state-weight pairs $(s,w)$ with $w\leq N_\varepsilon$. By Lemma \ref{lem:continuous}, we conclude that
	\[
		\VPE[\lambda]^{\tsched}_{\cM}\geq \VPE[\lambda]^{\sched_{k_{\varepsilon}}}_{\cM}-\varepsilon \geq  \VPE[\lambda]^{\sched}_{\cM}-\varepsilon .
	\]
	As $\varepsilon$ was arbitrary, this implies that  $\VPE[\lambda]^{\tsched}_{\cM}\geq \VPE[\lambda]^{\sched}$.
\end{proof}

\subsection{Hardness of the threshold problem}\label{app:hardness}

\theoremhardnessVPE*

\begin{proof}
	We reduce from the following problem which is shown to be EXPTIME-hard in \cite{HaaseKiefer15}:
	Given an MDP $\cM$ and a natural number $T>0$ such that $\goal$ is reached in~$\cM$ almost surely under all schedulers, decide whether
	there is a scheduler $\sched$ such that $\Pr^{\sched}_{\cM}(\rawdiaplus \goal{=T})=1$.
	Observe that~$\cM$ does not contain any end-components.

	Given such an MDP $\cM$ and a value $T>0$, we
	define the MDP~$\cM'$ by adding a fresh initial state $\iota$ from which a unique action (with weight~$0$) leads, with probability $1/2$, to
	the initial state of~$\cM$;
	while with probability $1/2$ it leads to a fresh state $\iota'$. From~$\iota'$, a unique action with weight~$T$ leads to $\goal$.

	Further, we let $\vartheta\eqdef T$.
	Let~$W$ be the largest weight in the MDP, and~$n$ the number of states,
	and $p_{\min}$ the smallest probability in~$\cM$.
	% Let~$m \in \mathbb{N}$ be a polynomially large constant with the following property:
	% \begin{itemize}
	% 	\item If there is a scheduler~$\sched$ such that~$\Pr^\sigma_{\cM}[\rawdiaplus\goal=T \mid \diamond \goal\leq m] = 1$,
	% 	      then $\Pr^{\max}(\rawdiaplus\goal=T)=1$, where $\diamond\goal$ is the number of steps before entering~$\goal$.
	% \end{itemize}
	% {\color{red}OS: This remains to be proven: I hope $m$ can be chosen to polynomially}
	Let~$\varepsilon<p_{\min}^n$.

	We establish an upper bound on the expectation of $\rawdiaplus \goal$ as follows.
	Since all schedulers are proper, under all schedulers, every~$n$ steps, the process reaches $\goal$ with probability at least~$\varepsilon$.
	We have, for all $\sched$,
	\begin{align}
		\mathbb{E}_{\cM'}^{\sched}(\rawdiaplus \goal ) & \leq W \cdot \mathbb{E}_{\cM'}^{\sched}(\text{steps until $\goal$})                    \nonumber     \\
		                                               & = W \sum_{k\geq 0} k \Pr_{\cM'}^{\sched}(\text{steps until $\goal$} =k) \nonumber        \\
		                                               & \leq W \sum_{k\geq 0} k (1-\varepsilon)^{\lfloor k/n\rfloor}            \nonumber  \\
		                                               & \leq W \sum_{l\geq 0} n(l+n) (1-\varepsilon)^l                           \nonumber \\
		                                               & \leq n W (\frac{n}{\varepsilon} + \frac{1-\varepsilon}{\varepsilon^2})
		%& \leq W (\frac{n}{\varepsilon} + \frac{1-\varepsilon}{\varepsilon^2}) = W\frac{1+(n-1)\varepsilon}{\varepsilon^2},
		\label{eqn:exptime-expect-bound}
	\end{align}
	where we used the fact that under~$\sched$, from any state, there is a path of length~$n$ to~$\goal$, with probability~$\geq \varepsilon$.
	Let us denote $f(n,W,\varepsilon) = n W (\frac{n}{\varepsilon} + \frac{1-\varepsilon}{\varepsilon^2})$ which we can assume to be larger than $1$.
	Define
	\[
		\lambda \eqdef 18 f(n,W,\varepsilon)
	\]
	Notice that this number can be computed in polynomial time.

	\medskip
	We claim that there exists a scheduler $\sched$ with $\Pr^{\sched}_{\cM}(\rawdiaplus\goal=T)=1$
	iff
	there exists $\sched'$ with $\VPE[\lambda]^{\sched'}_{\cM'} \geq \vartheta$.

	A scheduler $\sched$ with $\Pr^{\sched}_{\cM}(\rawdiaplus\goal=T)=1$ viewed as a scheduler for $\cM^\prime$ satisfies
	$\mathbb{E}^{\sched}_{\cM^\prime} (\rawdiaplus \goal )=T$ and $\Var^{\sched}_{\cM^\prime} (\rawdiaplus \goal )=0$. Hence,
	$\VPE[\lambda]^\sched_{\cM^\prime}=T=\vartheta$.

	Let $\sched$ be a deterministic scheduler for $\cM^\prime$ that maximizes $\VPE[\lambda]^{\sched}_{\cM^\prime}$,
	and assume that $\VPE[\lambda]^{\sched}_{\cM^\prime} \geq \vartheta$.
	Such a scheduler exists by the previous Theorem \ref{thm:determinist_scheduler} and can be viewed as a scheduler for $\cM$ as well.

	Let us write~$p = \Pr^{\sched}_{\cM}(\rawdiaplus \goal = T)$,
	and $\mu^\sched = \mathbb{E}_{\cM}^{\sched}(\rawdiaplus\goal)$.

	By assuming~$p < 1$, we will reach a contradiction.

	% \medskip
	% Let us start by showing that $p > \varepsilon$. Assume otherwise.
	% By contruction, $\cM'$ contains a path to~$\goal$ with weight~$T$ and probability $1/2$.
	% Moreover, $p\leq \varepsilon$ means that inside $\cM'$, $\rawdiaplus\goal \neq T$ is achieved with probability
	% $\geq 1/2(1-\varepsilon) \geq \frac{1}{3}$.
	% Thus, we have that either $|T - \mu^\sched|\geq \frac{1}{2}$ or $|l - \mu^\sched| \geq \frac{1}{2}$ for all $l \neq T$.
	% It follows that
	% \[
	% 	\Var^\sched_{\cM^\prime}(\rawdiaplus \goal)=
	% 	\mathbb{E}^\sched_{\cM^\prime}[(\rawdiaplus \goal - \mu^\sched)^2] \geq \frac{1}{3} \frac{1}{2},
	% \]
	% As $\mathbb{E}^{\sched}_{\cM^\prime}(\rawdiaplus \goal) \leq f(n,W,\varepsilon)$, we obtain
	% \(
	% \VPE[\lambda]^{\sched}_{\cM^\prime} \leq f(n,W,\varepsilon) - \lambda \frac{1}{6} < 0,
	% \)
	% which is a contradiction.
	\medskip

	We first show that $|T - \mu^\sched| \leq \frac{1}{3}$.
	In fact, observe that
	\begin{equation}
		\label{eqn:exptime-split-variance}
		\mathbb{V}_{\cM'}^{\sched}(\rawdiaplus\goal) = \frac{1}{2}(T - \mu^\sched)^2 + \frac{1}{2}\mathbb{E}_{\cM}^{\sched}((\rawdiaplus\goal-\mu^\sched)^2),
	\end{equation}
	where the first term corresponds to the path $\iota,\iota',\goal$, and the second term to all other paths.

	As $\mathbb{E}^{\sched}_{\cM^\prime}(\rawdiaplus \goal) \leq f(n,W,\varepsilon)$, we obtain
	\(
	\VPE[\lambda]^{\sched}_{\cM^\prime} \leq f(n,W,\varepsilon) - \lambda \cdot \frac{1}{2}(T-\mu^\sched)^2.
	\)
	So if $|T-\mu^\sched| > \frac{1}{3}$, then $\VPE[\lambda]^{\sched}_{\cM^\prime} \leq f(n,W,\varepsilon) - \frac{1}{18}\lambda<0$,
	which is a contradiction.

	\medskip
	Now, let us write
	\[
		\mathbb{E}_{\cM^\prime}^{\sched}(\rawdiaplus \goal)=\frac{1}{2}T + \frac{1}{2}p T
		+ \frac{1}{2}\sum_{k \neq T} k \Pr_{\cM}^{\sched}(\rawdiaplus\goal = k),
	\]
	where the first term corresponds to the path via $\iota,\iota',\goal$;
	the second term is the set of paths that enter~$\cM$ and achieve~$T$, and the last term
	is the contribution of all other paths.

	For the variance, let us focus on the right term of~\eqref{eqn:exptime-split-variance}.

	\begin{align}
		%\mathbb{V}^{\sched}_{\cM}(\rawdiaplus \goal) & = 
		\mathbb{E}_{\cM}^{\sched}((\rawdiaplus\goal - \mu^\sched)^2)
		 & = \Pr^{\sched}_{\cM}(\rawdiaplus\goal = T) \mathbb{E}^{\sched}_{\cM}((\rawdiaplus\goal -\mu^\sched)^2 \mid \rawdiaplus \goal = T) \nonumber             \\
		 & \qquad + \Pr^{\sched}_{\cM}(\rawdiaplus\goal \neq T)\mathbb{E}^{\sched}_{\cM}((\rawdiaplus\goal -\mu^\sched)^2 \mid \rawdiaplus \goal \neq T) \nonumber \\
		 & = p( T - \mu^\sched)^2 +
		(1-p)\mathbb{E}^{\sched}_{\cM}((\rawdiaplus -\mu^\sched)^2 \mid \rawdiaplus \goal \neq T) \nonumber                                                               \\
		 & = p(T-\mu^\sched)^2 + \sum_{l \neq T} (l-\mu^\sched)^2 \Pr^{\sched}_{\cM}(\rawdiaplus \goal = l).
		\label{eqn:exptime-var}
	\end{align}
	Here, the right hand side of the last line is obtained as follows.
	\[
		\begin{array}{ll}
			\mathbb{E}^{\sched}_{\cM}((\rawdiaplus\goal -\mu^\sched)^2 \mid \rawdiaplus \goal \neq T) & = \sum_{k \geq 0 } k \Pr_{\cM}^{\sched}((\rawdiaplus\goal - \mu^\sched)^2=k \mid \rawdiaplus \goal \neq T)  \\
			                                                                                          & = \sum_{l \geq 0} (l - \mu^\sched)^2 \Pr_{\cM}^{\sched}(\rawdiaplus\goal = l \mid \rawdiaplus \goal \neq T) \\
			                                                                                          & = \sum_{l \neq T} (l - \mu^\sched)^2 \Pr_{\cM}^{\sched}(\rawdiaplus\goal = l \mid \rawdiaplus \goal \neq T) \\
			                                                                                          & = \frac{1}{1-p} \sum_{l \neq T} (l - \mu^\sched)^2 \Pr_{\cM}^{\sched}(\rawdiaplus\goal = l).
		\end{array}
	\]
	So, we rewrite \eqref{eqn:exptime-split-variance} as follows.
	\begin{align*}
		\mathbb{V}_{\cM^\prime}^{\sched}(\rawdiaplus \goal) & = \frac{1+p}{2}(T - \mu^\sched)^2
		+ \frac{1}{2}\sum_{k \neq T} (k-\mu^\sched)^2 \Pr_{\cM}^{\sched}(\rawdiaplus\goal = k). \\
	\end{align*}

	Let us compute the VPE:
	\begin{align*}
		\VPE[\lambda]^{\sched}_{\cM^\prime} & =\mathbb{E}_{\cM^\prime}^{\sched}(\rawdiaplus \goal) - \lambda \mathbb{V}_{\cM^\prime}^{\sched}(\rawdiaplus \goal)                                                   \\
		                                    & = \frac{1+p}{2}T - \frac{\lambda(1+p)}{2}(T-\mu^\sched)^2 + \frac{1}{2}\sum_{k \neq T} (k - \lambda (k-\mu^\sched)^2) \Pr_{\cM}^{\sched}(\rawdiaplus\goal =k)
	\end{align*}
	To reach a contradiction, it suffices to show that the above is less than~$T$, which is equivalent to
	\begin{equation}
		\label{eqn:exptime-vpe-case2}
		- \lambda (1+p)(T-\mu^\sched)^2 + \sum_{k \neq T} (k - \lambda (k-\mu^\sched)^2) \Pr_{\cM}^{\sched}(\rawdiaplus\goal =k) < (1-p)T.
	\end{equation}
	The function $k \mapsto (k - \lambda (k-\mu^\sched)^2)$ is increasing until $k=\mu^\sched + \frac{1}{2\lambda}$ and decreasing afterwards;
	so its maximum at integers $k \neq T$ is reached either at $T-1$ or~$T+1$ since $|T- \mu^\sched| \leq \frac{1}{3}$.
	For~$k=T-1$, we have
	\begin{align*}
		 & - \lambda (1+p)(T-\mu^\sched)^2 +  (T-1 - \lambda (T-1-\mu^\sched)^2) \sum_{k \neq T} \Pr_{\cM}^{\sched}(\rawdiaplus\goal =k) <(1-p)T \\
		 & \Leftrightarrow -\lambda (1+p)(T-\mu^\sched)^2 + (1-p)T -(1 + \lambda(T-1-\mu^\sched)^2)(1-p) < (1-p)T                                       \\
		 & \Leftrightarrow -\lambda (1+p)(T-\mu^\sched)^2 -(1 + \lambda(T-1-\mu^\sched)^2)(1-p) < 0.                                                    \\
		%& \leq - \lambda (1+p)(T-\mu^\sched)^2 + (\mu^\sched + \frac{1}{2\lambda} - \lambda(\mu^\sched+\frac{1}{2\lambda}-\mu^\sched)^2)\sum_{k \neq T} \Pr_{\cM}^{\sched}(\rawdiaplus\goal =k) \\
		%& \leq - \lambda (1+p)(T-\mu^\sched)^2 + (\mu^\sched + \frac{1}{4\lambda})(1-p)                                                                                                                \\
	\end{align*}
	For $k=T+1$, we have
	\begin{align*}
		 & - \lambda (1+p)(T-\mu^\sched)^2 +  (T+1 - \lambda (T+1-\mu^\sched)^2) \sum_{k \neq T} \Pr_{\cM}^{\sched}(\rawdiaplus\goal =k) <(1-p)T \\
		 & \Leftrightarrow -\lambda (1+p)(T-\mu^\sched)^2 + (1-p)(T+1-\lambda(T+1-\mu^\sched)^2) < (1-p)T                                               \\
		 & \Leftrightarrow -\lambda (1+p)(T-\mu^\sched)^2 + (1-p)(1 - \lambda(T+1-\mu^\sched)^2) < 0.                                                   \\
	\end{align*}
	Here, $(T+1-\mu^\sched)^2\geq \frac{4}{9}$ since $|T-\mu^\sched|\leq \frac{1}{3}$,
	so $1 - \lambda(T+1-\mu^\sched)^2 < 0$ since $\lambda > \frac{9}{4}$.
	This establishes \eqref{eqn:exptime-vpe-case2}, yielding a contradiction.

        Addressing the second claim:
        For acyclic MDPs, the problem we reduce from is shown to be PSPACE-hard in  \cite{HaaseKiefer15}. As our construction preserves acyclicity, 
        PSPACE-hardness for the threshold problem for VPE in acyclic MDPs follows as above.
\end{proof}

\subsection{Saturation point}

\theoremsaturation*

\begin{proof}

	Let $\sched$ be a scheduler with $\VP[\lambda]^{\sched}_{\cM} = \VP[\lambda]^{\max}_{\cM} $.
	Suppose there is a $\sched$-path $\pi^\prime$ with $\wgt(\pi^\prime)\geq K$ such that
	\[
		\mathbb{E}^{\after{\sched}{\pi^\prime}}_{\cM,\last(\pi^\prime)} (\rawdiaplus \goal) > \mathbb{E}^{\min}_{\cM,\last(\pi^\prime)} (\rawdiaplus \goal).
	\]
	Then, there must be an $\sched$-path $\pi$ that extends $\pi^{\prime}$ such that $\sched$ chooses an action $\alpha\not\in \Act^{\min}(\last(\pi))$ with positive probability. Let us write $s\eqdef \last(\pi)$ and define
	\[p\eqdef P(\pi)\cdot \sched(\pi)(\alpha).
	\]
	So, $p$ is the probability that $\pi$ is seen under $\sched$ and that $\sched$ chooses $\alpha$ afterwards.
	As $\wgt(\pi)\geq K \geq B_{1/2}$, we conclude that $p\leq 1/2$.

	We claim that we can construct a scheduler $\sched^\prime$ with $\VP[\lambda]^{\sched^\prime}_{\cM}>\VP[\lambda]^{\sched}_{\cM}$.
	Let $\Min$ be a memoryless deterministic scheduler with
	\[
		\mathbb{E}^{\Min}_{\cM,s}(\rawdiaplus\goal) = \mathbb{E}^{\min}_{\cM,s}(\rawdiaplus\goal).
	\]
	We define the scheduler $\sched^\prime$ to behave like $\sched$ unless $\pi$ is seen and $\sched$ chooses $\alpha$ after $\pi$. In this case, $\sched^\prime$ switches to the behavior of $\Min$ instead.

	To compare the schedulers $\sched$ and $\sched^\prime$, let us define $\tsched$ to be the residual scheduler of $\sched$ after $\pi$ when $\sched$ chooses $\alpha$.
	I.e., extending the notation for residual schedulers, we define $\tsched\eqdef \after{\sched}{(\pi\alpha)}$ where
	\[
		\after{\sched}{(\pi\alpha)}(s)(\alpha)=1,
	\]
	i.e., on the path only consisting of state $s$, the scheduler chooses $\alpha$ with probability $1$, and for all finite paths $\rho$ starting with $s$ followed by $\alpha$,
	\[
		\after{\sched}{(\pi\alpha)}(\rho)=\sched(\pi\circ \rho).
	\]

	So, the schedulers $\sched$ and $\sched^\prime$ agree on all paths except for the extensions of $\pi$ in case $\sched$ chooses $\alpha$ after $\pi$. In this case, $\sched$ behaves like $\tsched$ and $\sched^\prime$ behaves like $\Min$.

	Let now $X$ be the event that $\sched$ does not choose $\alpha$ after $\pi$. In particular, $X$ contains all paths that do not have $\pi$ as a prefix.
	Conditioned on the event $X$, $\sched$ and $\sched^\prime$ behave identically. Furthermore, the probability of $X$ is $1-p$ under both schedulers.
	This allows us to split the expected value of $\rawdiaplus\goal$ under both schedulers conditioning on $X$ and its complement $\bar{X}$. We get
	\begin{align*}
		\mathbb{E}^{\sched}_{\cM,\sinit}(\rawdiaplus\goal)
		 & =(1-p) \mathbb{E}^{\sched}_{\cM,\sinit}(\rawdiaplus\goal\mid X) + p\cdot \mathbb{E}^{\sched}_{\cM,\sinit}(\rawdiaplus\goal\mid \bar{X})    \\
		 & =(1-p) \mathbb{E}^{\sched}_{\cM,\sinit}(\rawdiaplus\goal\mid X) + p\cdot(\wgt(\pi)+\mathbb{E}^{\tsched}_{\cM,s}(\rawdiaplus\goal)) \tag{5}
	\end{align*}
	and
	\begin{align*}
		\mathbb{E}^{\sched^\prime}_{\cM,\sinit}(\rawdiaplus\goal)
		 & =(1-p) \mathbb{E}^{\sched^\prime}_{\cM,\sinit}(\rawdiaplus\goal\mid X) + p\cdot \mathbb{E}^{\sched^\prime}_{\cM,\sinit}(\rawdiaplus\goal\mid \bar{X}) \\
		 & =(1-p) \mathbb{E}^{\sched}_{\cM,\sinit}(\rawdiaplus\goal\mid X) + p\cdot(\wgt(\pi)+\mathbb{E}^{\Min}_{\cM,s}(\rawdiaplus\goal)) \tag{6}
	\end{align*}
	where we use that $\sched$ and $\sched^\prime$ behave identically on $X$ in the last equality. Let us denote the term that does not depend on $\tsched$ or $\Min$ and occurs in both equations (5) and (6) by
	\[
		C_1\eqdef (1-p) \mathbb{E}^{\sched}_{\cM,\sinit}(\rawdiaplus\goal\mid X).
	\]
	Note that $C_1\leq \mathbb{E}^{\sched}_{\cM,\sinit}(\rawdiaplus\goal) \leq U_1$.

	As a direct consequence of (5) and (6), we also get
	\begin{align*}
		      & (\mathbb{E}^{\sched}_{\cM,\sinit}(\rawdiaplus\goal))^2                            \\
		=\,\, & C_1^2 +2C_1\cdot p\cdot(\wgt(\pi)+\mathbb{E}^{\tsched}_{\cM,s}(\rawdiaplus\goal))
		+ p^2\cdot(\wgt(\pi)+\mathbb{E}^{\tsched}_{\cM,s}(\rawdiaplus\goal))^2 \tag{7}
	\end{align*}
	and
	\begin{align*}
		      & (\mathbb{E}^{\sched^\prime}_{\cM,\sinit}(\rawdiaplus\goal))^2                  \\
		=\,\, & C_1^2 +2C_1\cdot p\cdot(\wgt(\pi)+\mathbb{E}^{\Min}_{\cM,s}(\rawdiaplus\goal))
		+ p^2\cdot(\wgt(\pi)+\mathbb{E}^{\Min}_{\cM,s}(\rawdiaplus\goal))^2. \tag{8}
	\end{align*}

	Applying the same reasoning as above to the random variable $\rawdiaplus \goal^2$, we obtain
	\begin{align*}
		\mathbb{E}^{\sched}_{\cM,\sinit}(\rawdiaplus\goal^2)
		 & =(1-p) \mathbb{E}^{\sched}_{\cM,\sinit}(\rawdiaplus\goal^2\mid X) + p\cdot \mathbb{E}^{\sched}_{\cM,\sinit}(\rawdiaplus\goal^2\mid \bar{X})    \\
		 & =(1-p) \mathbb{E}^{\sched}_{\cM,\sinit}(\rawdiaplus\goal^2\mid X) + p\cdot\mathbb{E}^{\tsched}_{\cM,s}((\wgt(\pi)+\rawdiaplus\goal)^2) \tag{9}
	\end{align*}
	and
	\begin{align*}
		\mathbb{E}^{\sched^\prime}_{\cM,\sinit}(\rawdiaplus\goal^2)
		 & =(1-p) \mathbb{E}^{\sched^\prime}_{\cM,\sinit}(\rawdiaplus\goal^2\mid X) + p\cdot \mathbb{E}^{\sched^\prime}_{\cM,\sinit}(\rawdiaplus\goal^2\mid \bar{X}) \\
		 & =(1-p) \mathbb{E}^{\sched}_{\cM,\sinit}(\rawdiaplus\goal^2\mid X) +  p\cdot\mathbb{E}^{\Min}_{\cM,s}((\wgt(\pi)+\rawdiaplus\goal)^2). \tag{10}
	\end{align*}
	Again, we abbreviate the first term by
	\[
		C_2\eqdef (1-p) \mathbb{E}^{\sched}_{\cM,\sinit}(\rawdiaplus\goal^2\mid X)
	\]
	and note that $C_2\leq\mathbb{E}^{\sched}_{\cM,\sinit}(\rawdiaplus\goal^2) \leq U_2$.
	Further extending equations (9) and (10) and using the linearity of the expected value, we obtain
	\begin{align*}
		\mathbb{E}^{\sched}_{\cM,\sinit}(\rawdiaplus\goal^2)
		 & = C_2 + p \cdot(\wgt(\pi)^2 + 2\cdot \wgt(\pi) \cdot \mathbb{E}^{\tsched}_{\cM,s}(\rawdiaplus\goal) + \mathbb{E}^{\tsched}_{\cM,s}(\rawdiaplus\goal^2))
		\tag{11}
	\end{align*}
	and
	\begin{align*}
		\mathbb{E}^{\sched^\prime}_{\cM,\sinit}(\rawdiaplus\goal^2)
		 & = C_2 + p \cdot(\wgt(\pi)^2 + 2\cdot \wgt(\pi) \cdot \mathbb{E}^{\Min}_{\cM,s}(\rawdiaplus\goal) + \mathbb{E}^{\Min}_{\cM,s}(\rawdiaplus\goal^2)).
		\tag{12}
	\end{align*}

	The key to showing our claim now  lies in the fact that $\tsched$ starts in state $s$ by choosing actions $\alpha\not\in \Act^{\min}(s)$. So,
	\[
		\mathbb{E}^{\tsched}_{\cM,s}(\rawdiaplus\goal) \geq \mathbb{E}^{\Min}_{\cM,s}(\rawdiaplus\goal) + \delta. \tag{13}
	\]

	Putting everything together, we will now show that indeed $\VP[\lambda]^{\sched^\prime}_{\cM}>\VP[\lambda]^{\sched}_{\cM}$:

	\begin{align*}
		          & \VP[\lambda]^{\sched^\prime}_{\cM} - \VP[\lambda]^{\sched}_{\cM}                                                                                                                    \\
		= \,\,    & \mathbb{E}^{\sched^\prime}_{\cM,\sinit}(\rawdiaplus\goal) - \mathbb{E}^{\sched}_{\cM,\sinit}(\rawdiaplus\goal)                                                                      \\
		          & - \lambda \cdot (\mathbb{E}^{\sched^\prime}_{\cM,\sinit}(\rawdiaplus\goal^2) - \mathbb{E}^{\sched}_{\cM,\sinit}(\rawdiaplus\goal^2)  )                                              \\
		          & + \lambda \cdot ((\mathbb{E}^{\sched^\prime}_{\cM,\sinit}(\rawdiaplus\goal))^2 - (\mathbb{E}^{\sched}_{\cM,\sinit}(\rawdiaplus\goal))^2 )                                           \\
		= \,\,    & p\cdot (\mathbb{E}^{\Min}_{\cM,s}(\rawdiaplus\goal) - \mathbb{E}^{\tsched}_{\cM,s}(\rawdiaplus\goal))                                                                               \\
		          & -  \lambda \cdot p \cdot \big(2\cdot \wgt(\pi) \cdot \mathbb{E}^{\Min}_{\cM,s}(\rawdiaplus\goal) + \mathbb{E}^{\Min}_{\cM,s}(\rawdiaplus\goal^2)                                    \\
		          & \phantom{ -  \lambda \cdot p \cdot - - -}        -2\cdot \wgt(\pi) \cdot \mathbb{E}^{\tsched}_{\cM,s}(\rawdiaplus\goal) - \mathbb{E}^{\tsched}_{\cM,s}(\rawdiaplus\goal^2)\big)     \\
		          & + \lambda\cdot \big( 2C_1 \cdot p \cdot  (\mathbb{E}^{\Min}_{\cM,s}(\rawdiaplus\goal) - \mathbb{E}^{\tsched}_{\cM,s}(\rawdiaplus\goal))                                             \\
		          & \phantom{+ \lambda\cdot \big(} + p^2 \cdot (   (\wgt(\pi)+\mathbb{E}^{\Min}_{\cM,s}(\rawdiaplus\goal))^2 - (\wgt(\pi)+\mathbb{E}^{\tsched}_{\cM,s}(\rawdiaplus\goal))^2      )\big) \\
		=\,\,     & p\cdot (\mathbb{E}^{\Min}_{\cM,s}(\rawdiaplus\goal) - \mathbb{E}^{\tsched}_{\cM,s}(\rawdiaplus\goal))                                                                               \\
		          & -  \lambda \cdot p \cdot \big(\mathbb{E}^{\Min}_{\cM,s}(\rawdiaplus\goal^2)-\mathbb{E}^{\tsched}_{\cM,s}(\rawdiaplus\goal^2)\big)                                                   \\
		          & + 2 \lambda \cdot p \cdot \wgt(\pi) \cdot \big( \mathbb{E}^{\tsched}_{\cM,s}(\rawdiaplus\goal) -  \mathbb{E}^{\Min}_{\cM,s}(\rawdiaplus\goal) \big)                                 \\
		          & + 2 \lambda\cdot p \cdot    (\mathbb{E}^{\Min}_{\cM,s}(\rawdiaplus\goal) - \mathbb{E}^{\tsched}_{\cM,s}(\rawdiaplus\goal))                                                          \\
		          & + 2\lambda\cdot p^2  \cdot   \wgt(\pi) \cdot \big( \mathbb{E}^{\Min}_{\cM,s}(\rawdiaplus\goal) - \mathbb{E}^{\tsched}_{\cM,s}(\rawdiaplus\goal) \big)                               \\
		          & + \lambda\cdot p^2 \cdot \big( (\mathbb{E}^{\Min}_{\cM,s}(\rawdiaplus\goal))^2 -  (\mathbb{E}^{\tsched}_{\cM,s}(\rawdiaplus\goal))^2 \big)                                          \\
		\geq \,\, & -p\cdot U_1  - \lambda \cdot p \cdot U_2  - 2\lambda \cdot p \cdot U_1  - \lambda\cdot p^2 \cdot U_1^2                                                                              \\
		          & +2 \lambda \cdot p \cdot \wgt(\pi) \cdot \big( \mathbb{E}^{\tsched}_{\cM,s}(\rawdiaplus\goal) -  \mathbb{E}^{\Min}_{\cM,s}(\rawdiaplus\goal) \big)                                  \\
		          & - 2\lambda\cdot p^2  \cdot   \wgt(\pi) \cdot \big( \mathbb{E}^{\tsched}_{\cM,s}(\rawdiaplus\goal) - \mathbb{E}^{\Min}_{\cM,s}(\rawdiaplus\goal) \big)                               \\
		\geq\,\,  & -p\cdot U_1  - \lambda \cdot p \cdot U_2  - 2\lambda \cdot p \cdot U_1  - \lambda\cdot p^2 \cdot U_1^2                                                                              \\
		          & +  \lambda \cdot p \cdot \wgt(\pi) \cdot \big( \mathbb{E}^{\tsched}_{\cM,s}(\rawdiaplus\goal) -  \mathbb{E}^{\Min}_{\cM,s}(\rawdiaplus\goal) \big). \tag{14}
	\end{align*}
	In the last inequality, we use that $p\leq 1/2$. To show that the right hand side is greater than $0$, we use that $\mathbb{E}^{\tsched}_{\cM,s}(\rawdiaplus\goal) -  \mathbb{E}^{\Min}_{\cM,s}(\rawdiaplus\goal)\geq \delta$ and first obtain that
	\begin{align*}
		          & -p\cdot U_1  - \lambda \cdot p \cdot U_2  - 2\lambda \cdot p \cdot U_1  - \lambda\cdot p^2 \cdot U_1^2                                                   \\
		          & +  \lambda \cdot p \cdot \wgt(\pi) \cdot \big( \mathbb{E}^{\tsched}_{\cM,s}(\rawdiaplus\goal) -  \mathbb{E}^{\Min}_{\cM,s}(\rawdiaplus\goal) \big)       \\
		\geq \,\, & -p\cdot U_1  - \lambda \cdot p \cdot U_2  - 2\lambda \cdot p \cdot U_1  - \lambda\cdot p^2 \cdot U_1^2 +   \lambda \cdot p \cdot \wgt(\pi) \cdot \delta.
	\end{align*}
	Now,
	\[
		-p\cdot U_1  - \lambda \cdot p \cdot U_2  - 2\lambda \cdot p \cdot U_1  - \lambda\cdot p^2 \cdot U_1^2 +   \lambda \cdot p \cdot \wgt(\pi) \cdot \delta > 0
	\]
	is equivalent to
	\[
		\lambda  \cdot \wgt(\pi) \cdot \delta >   U_1 + \lambda U_2 + 2\lambda U_1 + \lambda \cdot p \cdot U_1^2.
	\]
	As $\wgt(\pi)\geq K$ and $p\leq 1/2$, it is sufficient to show that
	\[
		\lambda  \cdot K \cdot \delta >   U_1 + \lambda U_2 + 2\lambda U_1 + \lambda \cdot 1/2 \cdot U_1^2.
	\]
	This now, however follows directly from the definition of $K$ as
	\[
		\lambda  \cdot K \cdot \delta \geq U_1 + \lambda U_2 + 2\lambda U_1 + \lambda \cdot 1/2 \cdot U_1^2 + \lambda\cdot \delta.
	\]
	This shows that the scheduler $\sched$ does not maximize $\VP[\lambda]^{\sched}_{\cM}$.
	This finishes the proof that the defined value $K$, which is computable in polynomial time as argued before, is as desired.
	
	Note that the estimations (14) and hence the whole proof works analogously for the
	variant $\sup_\sched -\mathbb{E}^{\sched}_{\cM}(\rawdiaplus \goal) - \lambda\cdot \Var^{\sched}_{\cM}(\rawdiaplus \goal)$
	as claimed in Remark \ref{rem:minimal_expectation_VPE}.
\end{proof}

\theoremoptimalaboveK*

\begin{proof}
	We will show that after each $\sched$-path $\pi$ with $\wgt(\pi)\geq K$ behaving according to $\vsched$ instead of $\after{\sched}{\pi}$ does not decrease the variance-penalized expectation. Applying this reasoning successively to all  $\sched$-paths that reach a weight level of at least $K$ in their last step leads us to the desired scheduler $\tsched$.

	So, let $\pi$ be a $\sched$-path with $\wgt(\pi)\geq K$ ending in a state $s$.
	By Theorem \ref{thm:saturation}, we know that
	\[
		\mathbb{E}^{\after{\sched}{\pi}}_{\cM,s}(\rawdiaplus \goal) = \mathbb{E}^{\vsched}_{\cM,s}(\rawdiaplus \goal)
	\]
	and by the definition of $\vsched$,
	\[
		\Var^{\after{\sched}{\pi}}_{\cM,s}(\rawdiaplus \goal) \geq \Var^{\vsched}_{\cM,s}(\rawdiaplus \goal).
	\]
	As in previous proofs of this section, we express the variance-penalized expectation of $\sched$ by conditioning on the event $\Pi$ that $\pi$ is a prefix of a run and its complement $\neg \Pi$. Note that by the definition of $K$ and the fact that $\pi$ is a $\sched$-path, the probability $p\eqdef\Pr^{\sched}_{\cM}(\pi)$ lies strictly between $0$ and $1$. We obtain
	\begin{align*}
		\VPE[\lambda]^{\sched}_{\cM} = {} & \mathbb{E}^{\sched}_{\cM}(\rawdiaplus \goal) + \lambda\cdot (\mathbb{E}^{\sched}_{\cM}(\rawdiaplus \goal))^2 - \lambda\cdot \mathbb{E}^{\sched}_{\cM}(\rawdiaplus \goal^2)                                        \\
		{}={}                             &
		\mathbb{E}^{\sched}_{\cM}(\rawdiaplus \goal)
		+ \lambda\cdot (\mathbb{E}^{\sched}_{\cM}(\rawdiaplus \goal))^2                                                                                                                                                                                       \\
		                                  &
		- \lambda\cdot \big( (1-p)\mathbb{E}^{\sched}_{\cM}(\rawdiaplus \goal^2\mid \neg \Pi) + p \mathbb{E}^{\after{\sched}{\pi}}_{\cM,s} ((\wgt(\pi)+\rawdiaplus \goal)^2)\big)                                                                             \\
		{}={}                             &
		\mathbb{E}^{\sched}_{\cM}(\rawdiaplus \goal)
		+ \lambda\cdot (\mathbb{E}^{\sched}_{\cM}(\rawdiaplus \goal))^2                                                                                                                                                                                       \\
		                                  &
		- \lambda\cdot (1-p)\mathbb{E}^{\sched}_{\cM}(\rawdiaplus \goal^2\mid \neg \Pi)                                                                                                                                                                       \\
		                                  & - \lambda\cdot p \cdot  \big(  \wgt(\pi)^2 + 2\wgt(\pi)\cdot  \mathbb{E}^{\after{\sched}{\pi}}_{\cM,s} (\rawdiaplus \goal)+ \mathbb{E}^{\after{\sched}{\pi}}_{\cM,s} (\rawdiaplus \goal^2)\big). \tag{$\times$}
	\end{align*}
	Let now $\sched^\prime$ be the scheduler that behaves like $\sched$ after all paths that do not have $\pi$ as a prefix and that behaves like $\vsched$ as soon as $\pi$ has been seen. By the observations on $\after{\sched}{\pi}$ and $\vsched$ above, we know that
	$\mathbb{E}^{\sched}_{\cM}(\rawdiaplus\goal)=\mathbb{E}^{\sched^\prime}_{\cM}(\rawdiaplus\goal)$. Furthermore, conditioned on $\neg \Pi$, the two schedulers behave identically.
	With the same calculations for $\sched^\prime$ as in ($\times$) using that $\after{\sched^\prime}{\pi}=\vsched$, we get that
	\[
		\VPE[\lambda]^{\sched^\prime}_{\cM}\geq \VPE[\lambda]^{\sched}_{\cM}
	\]
	if and only if
	\[
		- \lambda\cdot p \cdot \mathbb{E}^{\vsched}_{\cM,s} (\rawdiaplus \goal^2) \geq - \lambda\cdot p \cdot \mathbb{E}^{\after{\sched}{\pi}}_{\cM,s} (\rawdiaplus \goal^2)
	\]
	This, however, follows directly from
	\begin{align*}
		\mathbb{E}^{\vsched}_{\cM,s} (\rawdiaplus \goal^2) = {} & \Var^{\vsched}_{\cM,s}(\rawdiaplus\goal) + (\mathbb{E}^{\vsched}_{\cM,s} (\rawdiaplus \goal))^2                         \\
		{}={}                                                   & \Var^{\vsched}_{\cM,s}(\rawdiaplus\goal) + (\mathbb{E}^{\after{\sched}{\pi}}_{\cM,s} (\rawdiaplus \goal))^2             \\
		{}\leq{}                                                & \Var^{\after{\sched}{\pi}}_{\cM,s}(\rawdiaplus\goal) + (\mathbb{E}^{\after{\sched}{\pi}}_{\cM,s} (\rawdiaplus \goal))^2 \\
		{}={}                                                   & \mathbb{E}^{\after{\sched}{\pi}}_{\cM,s} (\rawdiaplus \goal^2).
	\end{align*}
	Enumerating all $\sched$-paths $\pi_1,\pi_2,\dots$ that reach a weight of at least $K$ in their last step, we can now recursively define a sequence of schedulers $\sched_i$ starting from $\sched_1=\sched$ with non-decreasing variance-penalized expectation such that $\sched_i$ behaves like $\vsched$ after all paths $\pi_j$ with $j<i$. By continuity arguments as before, we obtain a scheduler $\tsched$ in the limit that behaves like $\vsched$ as soon as a weight of at least $K$ has been accumulated and that satisfies
	\[
		\VPE[\lambda]^{\tsched}_{\cM}\geq\VPE[\lambda]^{\sched}_{\cM}.
	\]
	By the optimality of $\sched$, the new scheduler $\tsched$ maximizes the variance-penalized expectation, too.
\end{proof}

\end{appendix}

%%%%%%%%%%%%%%%%%%%%%%%%%%%%%%%%%%%%%%%%%%%%%%%%%%%%%%%%%%%%%%%%%%%%%%

\end{document}